\documentclass[a4paper,reqno]{amsart}

\usepackage{graphicx}

\usepackage{amssymb}

\usepackage{amsmath,amsthm,enumerate}
\usepackage{color}
\usepackage[format=plain,font=small,width=0.3\textwidth]{caption}
\usepackage{wrapfig}
\usepackage[%
reset,%
papersize={210truemm,297truemm},%
vmargin={2.8truecm},%
hmargin=2.8truecm%
]{geometry}

\usepackage{ulem} 

\makeatletter
\def\@settitle{\begin{center}\baselineskip14\p@\relax\bfseries{\large\@title}\thispagestyle{empty}\end{center}}
\def\@setauthors{%
  \begingroup
  \def\thanks{\protect\thanks@warning}%
  \trivlist
  \centering\footnotesize \@topsep30\p@\relax
  \advance\@topsep by -\baselineskip
  \item\relax
  \author@andify\authors
  \def\\{\protect\linebreak}%
  {\authors}%
  \ifx\@empty\contribs
  \else
    ,\penalty-3 \space \@setcontribs
    \@closetoccontribs
  \fi
  \endtrivlist
  \endgroup
}
\def\maketitle{\par
  \@topnum\z@ 
  \@setcopyright
  \thispagestyle{firstpage}
  \ifx\@empty\shortauthors \let\shortauthors\shorttitle
  \else \andify\shortauthors
  \fi
  \@maketitle@hook
  \begingroup
  \@maketitle
  \toks@\@xp{\shortauthors}\@temptokena\@xp{\shorttitle}%
  \toks4{\def\\{ \ignorespaces}}
  \edef\@tempa{%
    \@nx\markboth{\the\toks4
      \@nx
      {\the\toks@}}{\the\@temptokena}}%
  \@tempa
  \endgroup
  \c@footnote\z@
  \@cleartopmattertags
}
\makeatother

\newcommand{\out}[1]{}

\newtheorem{Proposition}{Proposition}[section]
\newtheorem{Theorem}[Proposition]{Theorem}
\newtheorem{Lemma}[Proposition]{Lemma}

\theoremstyle{definition}

\newtheorem{Definition}[Proposition]{Definition}
\newtheorem{Remark}[Proposition]{Remark}

\newcommand{\set}[2]{\left\{#1 ;\; #2 \right\}}

\newcommand{\pmat}[1]{\begin{pmatrix}#1\end{pmatrix}}

\newcommand{\R}{\mathbb{R}}

\newcommand{\C}{\mathbb{C}}
\newcommand{\Z}{\mathbb{Z}}

\newcommand{\CS}[2]{\mathrm{CS}_{#1}(#2)}
\newcommand{\domainF}[1][f_{\kappa,\gamma}]{\mathcal{D}(#1)}

\newcommand{\sgn}{\mathsf{s}}

\newcommand{\ds}{\displaystyle}

\newcommand{\bl}[1]{\textcolor{blue}{#1}}

\begin{document}
\title{
Analysis on Lambert-Tsallis functions
}


\author{Hideto Nakashima}
\address[Hideto Nakashima]{Graduate School of Mathematics, Nagoya University, Furo-cho, Chikusa-ku, Nagoya 464-8602, Japan}
\email{h-nakashima@math.nagoya-u.ac.jp}  
\author{Piotr Graczyk}
\address[Piotr Graczyk]{Laboratoire de Math{\'e}matiques LAREMA, Universit{\'e} d'Angers 2, boulevard Lavoisier, 49045 Angers Cedex 01, France}
\email{piotr.graczyk@univ-angers.fr}

\begin{abstract}
    In this paper, we study the Lambert-Tsallis $W$ function, which is a generalization of the Lambert $W$ function with two real parameters.
    We give a condition on the parameters such that there exists a complex domain touching zero on boundary
    which is mapped bijectively to the upper half plane by the Lambert-Tsallis function.
\end{abstract}

\maketitle

\section*{Introduction}
The Lambert $W$ function, which is the multivalued inverse of the function $z\mapsto ze^z$ has many applications in many areas of mathematics.
In a theory of random matrices,
it appears in a formula of a eigenvalue distribution of a certain Wishart-type ensemble (cf.\ Cheliotis~\cite{Cheliotis}).
In the previous paper~\cite{NG2020_1},
we found that,
for a certain class of Wishart-type ensembles,
the corresponding eigenvalue distributions can be described by using the main branch of the inverse function of $z\mapsto \frac{z}{1+\gamma z}\bigl(1+\frac{z}{\kappa}\bigr)^\kappa$, where $\gamma,\kappa$ are real parameters.
The aim of this paper is to give a complete analysis of this function, which we call the Lambert-Tsallis function.

\section{Preliminaries}

For a non zero real number $\kappa$, we set
\[
\exp_\kappa(z):=\left(1+\frac{z}{\kappa}\right)^\kappa\quad(1+\frac{z}{\kappa}\in\C\setminus\R_{\le 0}),
\]
where we take the main branch of the power function when $\kappa$ is not integer.
If $\kappa=\frac{1}{1-q}$, then it is exactly the so-called Tsallis \textit{$q$-exponential function}
(cf.\ \cite{AmariOhara2011, ZNS2018}).
By virtue of
$\ds\lim_{\kappa\to\infty}{\exp_\kappa(z)}=e^z$,
we regard $\exp_\infty(z)=e^z$.

For two real numbers $\kappa,\gamma$ such that $\kappa\ne 0$,
we introduce a holomorphic function 
$f_{\kappa,\gamma}(z)$,
which we call \textit{generalized Tsallis function}, by
\[
f_{\kappa,\gamma}(z):=\frac{z}{1+\gamma z}\exp_\kappa(z)\quad(1+\frac{z}{\kappa}\in\C\setminus\R_{\le 0}).
\]
Analogously to Tsallis $q$-exponential, 
we also consider $f_{\infty,\gamma}(z)=\frac{ze^z}{1+\gamma z}$ $(z\in\C)$.
In particular, $f_{\infty,0}(z)=ze^z$.
Let $\domainF$ be the domain of $f_{\kappa,\gamma}$, that is, 
if $\kappa$ is integer then $\domainF=\C\setminus\{-\frac1\gamma\}$
if $\kappa$ is not integer, then 
\[\domainF=\C\setminus\set{x\in\R}{1+\frac{x}{\kappa}\le 0,\text{ or }x=-\frac{1}{\gamma}}.\]

The purpose of this work is to study an inverse function to $f_{\kappa,\gamma}$ in detail.  
A multivariate inverse function of $f_{\infty,0}(z)=ze^z$ is called 
the Lambert $W$ function and studied
in \cite{Corless}. 
Hence, we call an inverse function to $f_{\kappa,\gamma}$ 
the {\it Lambert--Tsallis $W$ function.}
Since we have 
\begin{equation}\label{eq:derivative f}
f'_{\kappa,\gamma}(z)=\frac{\gamma z^2+\bigl(1+1/\kappa\bigr)z+1}{(1+\gamma z)^2}\left(1+\frac{z}{\kappa}\right)^{\kappa-1},
\end{equation}
the function $f_{\kappa,\gamma}(z)$ has the inverse function $w_{\kappa,\gamma}$ 
in a  neighborhood of $z=0$
by the fact $f'_{\kappa,\gamma}(0)=1\ne 0$.

Let $z=x+yi\in\C$ and we set for $\kappa\ne \infty$
\[
\theta(x,y):=\mathrm{Arg}\Bigl(1+\frac{z}{\kappa}\Bigr),
\]
where $\mathrm{Arg}(w)$ stands for the principal argument of $w$; $-\pi<\mathrm{Arg}(w)\le \pi$.
Since we now take the main branch of power function,
we have for $\kappa\ne \infty$
\[
\begin{array}{r@{\ }c@{\ }l}
\ds
\left(1+\frac{z}{\kappa}\right)^\kappa
&=&
\ds
\exp\left(
\kappa\left( \log\left|1+\frac{z}{\kappa}\right|+i{\mathrm{Arg}}\left(1+\frac{z}{\kappa}\right)\right)
\right)
=
\left(\Bigl(1+\frac{x}{\kappa}\Bigr)^2+\frac{y^2}{\kappa^2}\right)^{\tfrac{\kappa}{2}}
e^{i\kappa\theta(x,y)}\\
&=&
\ds
\left(\Bigl(1+\frac{x}{\kappa}\Bigr)^2+\frac{y^2}{\kappa^2}\right)^{\tfrac{\kappa}{2}}
\Bigl(\cos(\kappa\theta(x,y))+i\sin(\kappa\theta(x,y))\Bigr).
\end{array}
\]
If $\kappa=\infty$, then we regard $\kappa\theta(x,y)$ as $y$
because we have $\ds\lim_{\kappa\to+\infty}\exp_\kappa(z)=e^z=e^x(\cos y+i\sin y)$.
Since
\[
\frac{z}{1+\gamma z}
=
\frac{z(1+\gamma\bar z)}{{|1+\gamma z|^2}}
=
\frac{(x+\gamma x^2+\gamma y^2)+i(y+\gamma xy-\gamma xy)}{(1+\gamma x)^2+\gamma^2 y^2}
=
\frac{(x+\gamma x^2+\gamma y^2)+iy}{(1+\gamma x)^2+\gamma^2 y^2},
\]
we have
\begin{equation}\label{eq:calc f}
f_{\kappa,\gamma}(z)
=
\ds
\frac{
\Bigl((1+x/\kappa)^2+(y/\kappa)^2\Bigr)^{\tfrac{\kappa}{2}}
}{(1+\gamma x)^2+\gamma^2 y^2}
\left(
\begin{array}{@{\,}l@{\,}}
(x+\gamma x^2+\gamma y^2)\cos(\kappa\theta(x,y)) - y\sin(\kappa\theta(x,y))\\
\quad+i\bigl\{
    (x+\gamma x^2+\gamma y^2)\sin(\kappa\theta(x,y)) + y\cos(\kappa\theta(x,y)) 
\bigr\}
\end{array}
\right).
\end{equation}
Then, $f_{\kappa,\gamma}(z)\in\R$ implies
\begin{equation}\label{eq:implicit}
(x+\gamma x^2+\gamma y^2)\sin(\kappa\theta(x,y)) + y\cos(\kappa\theta(x,y))=0.
\end{equation}
If $\sin(\kappa\theta(x,y))=0$, then $\cos(\kappa\theta(x,y))$ does not vanish so that $y$ needs to be zero.
Thus, if $z=x+yi\in f_{\kappa,\gamma}(\R)$ with $y\ne 0$, then we have $\sin(\kappa\theta(x,y))\ne 0$.
Thus,
the equation \eqref{eq:implicit} for $\sin(\kappa\theta(x,y))\ne 0$ can be rewritten as
\begin{equation}
    \label{im}
    F(x,y):=(x+\gamma x^2+\gamma y^2) + y\cot(\kappa\theta(x,y))=0.
\end{equation}
For $y=0$, we set $\ds F(x,0):=\lim_{y\to0}F(x,y)$.
If $x=0$, then we have $\theta(0,y)=\mathrm{Arctan}(\frac{y}{\kappa})$
and hence $\ds F(0,0)=\lim_{y\to0}y\cot(\kappa\mathrm{Arctan}(\tfrac{y}{\kappa}))=1$.
Let us introduce the connected set $\Omega=\Omega_{\kappa,\gamma}$ by
\[\Omega=\Omega_{\kappa,\gamma}:=\set{z=x+yi\in \domainF}{F(x,y)>0}^\circ,\] 
where $A^\circ$ is the connected component of an open set $A\subset \C$ containing $z=0$.
Note that since $F$ is an even function on $y$,
the domain $\Omega$ is symmetric with respect to the real axis.
Set
\begin{equation}
    \label{def:mathcal S}
    \mathcal{S}:=\R\setminus f_{\kappa,\gamma}\bigl(\Omega_{\kappa,\gamma}\cap \R\bigr).
\end{equation}

\begin{Definition}
If there exists a unique holomorphic extension $W_{\kappa,\gamma}$
of $w_{\kappa,\gamma}$ to $\C \setminus \overline{\mathcal{S}}$,
then we call $W_{\kappa,\gamma}$
{\it the main branch} of Lambert-Tsallis $W$ function.
In this paper, 
we only study and use $W_{\kappa,\gamma}$ among other branches
so that we call $W_{\kappa,\gamma}$ the \textit{Lambert--Tsallis function} for short.
Note that in our terminology the Lambert-Tsallis  $W$ function  is multivalued
and the
Lambert-Tsallis function $W_{\kappa,\gamma}$ is single-valued.
\end{Definition}

Our goal is to prove the following theorem.

\begin{Theorem}
\label{theorem:lambert-tsallis}
Let $f_{\kappa,\gamma}$ be a generalized Tsallis function.
Then, 
there exists the main branch of Lambert-Tsallis function
if and only if
{\rm(i)} $0<\kappa<1$ and $\gamma\le 0$,
{\rm(ii)} $\kappa\ge 1$ and $\gamma\le\frac14(1+\frac1\kappa)^2$,
{\rm(iii)} $-1<\kappa<0$ and $\gamma\le \frac1\kappa$,
{\rm(iv)} $\kappa\le -1$ and $\gamma\le\frac14(1+\frac1\kappa)^2$,
and 
{\rm(v)} $\kappa=\infty$ and $\gamma\le \frac14$.
\end{Theorem}

\begin{Remark}
In the case $\kappa>1$ and $D(0)<0$, the function $f_{\kappa,\gamma}$ maps $\Omega\cap\C^+$ to $\C^+$ two-to-one,
and hence the extension exists on a smaller domain in $\Omega$ 
which $f_{\kappa,\gamma}$ maps to $\C\setminus\overline{\mathcal{S}}$, 
but it is not unique. 
In the case $0<\kappa<1$ and $\gamma>0$, we need to extend the defining domain of $f_{\kappa,\gamma}$
so that there does not exists a complex domain such that $f_{\kappa,\gamma}$ maps to $\C\setminus\overline{\mathcal{S}}$.
\end{Remark}

The proof of this theorem is done in two steps,
that is,
we first give an explicit expression of $\Omega=\Omega_{\kappa,\gamma}$,
and then show that $f_{\kappa,\gamma}$ maps $\Omega$ to $\C\setminus\overline{\mathcal{S}}$ bijectively.
At first, we suppose that $0<\kappa<+\infty$.
Let us change variables in \eqref{im} by
\begin{equation}\label{eq:v change}
re^{i\theta}=1+\frac{z}{\kappa}
\quad(r>0,\ \theta\in(0,\pi)),
\quad\text{or equivalently}\quad
\begin{cases}
x=\kappa(r\cos\theta-1),\\
y=\kappa r\sin \theta,
\end{cases}
\end{equation}
and set  $a:=\gamma\kappa$ and
\begin{equation}\label{eq:b}
b(\theta)=(1-2a)\cos\theta+\sin\theta\cot(\kappa\theta).
\end{equation}
Then,
the equation \eqref{im} can be written as
\begin{equation}\label{eq:r}
a r^2+b(\theta)r+a-1=0.
\end{equation}
In fact, we have
\[
\begin{array}{cl}
&\ds
\kappa(r\cos\theta-1)
+
\gamma\bigl\{(\kappa(r\cos\theta-1))^2+(\kappa r\sin \theta)^2\bigr\}
+\kappa r\sin\theta \cot(\kappa\theta)=0\\
\Longleftrightarrow
&
\gamma\kappa^2 r^2+\bigl\{\kappa\cos\theta -2\gamma \kappa^2\cos\theta+\kappa\sin\theta\cot(\kappa\theta) \bigr\}r+(\gamma\kappa^2-\kappa)=0\\
\Longleftrightarrow&
\ds
\gamma\kappa r^2+
\Bigl\{(1 -2\gamma \kappa)\cos\theta+\sin\theta\cot(\kappa\theta)\Bigr\}r+\gamma\kappa-1
=0.
\end{array}
\]
If $\sin(\kappa\theta)\ne 0$,
then \eqref{eq:r} has a solution
\[
r=r_{\pm}(\theta)=\frac{-b(\theta)\pm\sqrt{b(\theta)^2-4a(a-1)}}{2a}.
\]
Thus, for each angle $\theta$,
there exists at most two points on $f_{\kappa,\gamma}^{-1}(\R)$.
Since the change \eqref{eq:v change} of variables is the polar transformation,
we need to know whether $r_\pm(\theta)$ is positive real or not.
We note that
\begin{equation}\label{eq:r prime}
r_{{\varepsilon}}'(\theta)=\frac{1}{2a}\left(
-b'(\theta)+\varepsilon\frac{2b(\theta)b'(\theta)}{2\sqrt{D(\theta)}}
\right)
=
\frac{-\varepsilon b'(\theta)}{2a}\cdot\frac{-b(\theta)+\varepsilon\sqrt{D(\theta)}}{\sqrt{D(\theta)}}
=
-\varepsilon b'(\theta)\frac{r_\varepsilon(\theta)}{\sqrt{D(\theta)}}
\end{equation}
for $\varepsilon=\pm1$.
Set 
$D(\theta):=b(\theta)^2-4a(a-1)$.
Then, $r_\pm(\theta)$ are real if and only if $D(\theta)\ge 0$,
and we have
\[
D'(\theta)=2b(\theta)b'(\theta).
\]

Let $\alpha_1,\ \alpha_2$ be the two solutions of 
\begin{equation}\label{eq:quad}
q(z):=\gamma z^2+\bigl(1+1/\kappa\bigr)z+1=0.
\end{equation}
If $\alpha_i$ are real, then we assume that $\alpha_1\le \alpha_2$,
and if not, then we assume that $\mathrm{Im}\,\alpha_1>0$ and $\mathrm{Im}\,\alpha_2<0$.
Then, $f'(z)=0$ implies $z=\alpha_i$ $(i=1,2)$ or $z=-\kappa$ if $\kappa>1$.
It is clear that $\alpha_1,\alpha_2$ are real numbers if and only if
\[
\Bigl(1+\frac{1}{\kappa}\Bigr)^2-4\gamma\ge 0
\iff
\gamma\le \frac14\Bigl(1+\frac{1}{\kappa}\Bigr)^2.
\]
These two points $\alpha_i$, $i=1,2$ correspond to the solutions of \eqref{eq:r} with parameter $\theta=0$.
In fact,
if $\theta=0$, then $r=1+\frac{x}{\kappa}$ and
\begin{equation}\label{def:b zero}
b(0):=\lim_{\theta\to0}b(\theta)=\frac{\kappa+1}{\kappa}-2a,
\end{equation}
and hence the equation \eqref{eq:r} with $\theta=0$ is described as 
\[
\begin{array}{r@{\ }c@{\ }l}
0
&=&
a r^2+b(0)r+a-1=a(r-1)^2+(1+\tfrac{1}{\kappa})(r-1)+\tfrac{1}{\kappa}\\
&=&
\ds
a\cdot\frac{x^2}{\kappa^2}+(1+\tfrac{1}{\kappa})\frac{x}{\kappa}+\frac{1}{\kappa}
=
\frac1\kappa\Bigl(\gamma x^2+\Bigl(1+\frac1\kappa\Bigr)x+1\Bigr).
\end{array}
\]
We note that 
\begin{equation}
\label{eq:q special values}
q\Bigl(-\frac1\gamma\Bigr)=\frac{a-1}{a},\quad
q(-\kappa)=(a-1)\kappa,
\end{equation}
and
\[
D(0)=b(0)^2-4a(a-1)=\Bigl(1+\frac1\kappa\Bigr)^2-\frac{4a}{\kappa}.
\]

In order to know whether the equation \eqref{eq:r} has a positive solution,
we need to study the signature of $D(\theta)$, and hence that of $b(\theta)$ and $b'(\theta)$.
To do so,
we introduce three functions $H_\alpha$, $F_\kappa$ and $J_\kappa$ as below, 
and investigate them in detail.

For $\alpha>0$, we set
\[
H_\alpha(x):=\sin(\alpha x)-\alpha \sin(x)\quad(x\in\R).
\]
Let us investigate the signature of $H_\alpha(x)$ on the interval $I_1=(0,\min(\frac{2\pi}{\alpha},2\pi))$.
Note that if $\alpha=1$ then we have $H_1\equiv 0$,
and hence we exclude the case $\alpha=1$.
By differentiating, we have
\[
H'_\alpha(x)=\alpha\cos(\alpha x)-\alpha\cos x
=
-2\alpha\sin\Bigl(\frac{\alpha+1}{2}x\Bigr)\sin\Bigl(\frac{\alpha-1}{2}x\Bigr).
\]
$H'_\alpha(x)=0$ implies that $x=\frac{2n\pi}{\alpha+1}$ $(n\in\Z)$ or $x=\frac{2m\pi}{\alpha-1}$ $(m\in\Z)$.

\begin{Lemma}
\label{lemma:func_H}
\begin{enumerate}[\rm(1)]
    \item If $0<\alpha\le \frac12$, then one has $H_\alpha(x)>0$ for any $x\in I_1$.
    \item If $\frac12<\alpha<1$, then there exists a unique $y_*$ such that $H_\alpha(y_*)=0$,
    and one has $H_\alpha(x)>0$ for $x\in(0,y_*)$ and $H_\alpha(x)<0$ for $x\in(y_*,2\pi)$.
    \item If $1<\alpha<2$, then there exists a unique $y_*\in I_1$ such that $H_\alpha(y_*)=0$,
    and one has $H_\alpha(x)<0$ for $x\in(0,y_*)$ and $H_\alpha(x)>0$ for $x\in(y_*,\frac{2\pi}{\alpha})$.
    \item If $\alpha\ge 2$, then one has $H_\alpha(x)<0$ for any $x\in I_1$.
\end{enumerate}
\end{Lemma}

\begin{proof}
It is obvious that $H_\alpha(0)=0$.
First, we assume that $0<\alpha <1$.
In this case,
since $\frac{1}{1-\alpha}\le 1$ and $\frac12< \frac{1}{\alpha+1}<1$,
we have 
\[
0<\frac{2\pi}{\alpha+1}<2\pi<\min\Bigl(\frac{2\pi}{\alpha+1},\frac{-2\pi}{\alpha-1}\Bigr).
\]
This means that $H'_\alpha(x)>0$ when $x\in(0,\frac{2\pi}{\alpha+1})$, and $H'_\alpha(x)<0$ when $x\in(\frac{2\pi}{\alpha+1},2\pi)$.
On the other hand,
we have
$H_\alpha(2\pi)=\sin(2\pi\alpha)-\alpha\sin(2\pi)=\sin(2\pi\alpha)$.
If $0<\alpha\le \frac12$,
then we have $0<2\alpha\pi\le\pi$ and thus $H_\alpha(2\pi)\ge 0$,
and if $\frac12<\alpha<1$ then we have $\pi<2\pi\alpha<2\pi$ which implies $H_\alpha(2\pi)<0$.
Thus, we obtain the assertions (1) and (2).

Next we assume that $\alpha>1$.
In this case, since $\frac{2}{\alpha+1}>\frac1\alpha$, we have
\[
0<\frac{2\pi}{\alpha+1}<\frac{2\pi}{\alpha}<\min\Bigl(\frac{4\pi}{\alpha+1},\frac{2\pi}{\alpha-1}\Bigr).
\]
This means that $H'_\alpha(x)<0$ when $x\in(0,\frac{2\pi}{\alpha+1})$, 
and $H'_\alpha(x)>0$ when $x\in(\frac{2\pi}{\alpha+1},\frac{2\pi}{\alpha})$.
On the other hand,
we have
$H_\alpha(\frac{2\pi}{\alpha})=\sin(2\pi)-\alpha\sin(\frac{2\pi}{\alpha})=-\alpha\sin(\frac{2\pi}{\alpha})$.
If $1<\alpha<2$, then we have $\pi<\frac{2\pi}{\alpha}<2\pi$ and hence $H_\alpha(\frac{2\pi}{\alpha})>0$.
If $\alpha\ge 2$, then we have $0<\frac{2\pi}{\alpha}\le \pi$ so that $H_\alpha(\frac{2\pi}{\alpha})\le 0$.
Therefore, we have proved the assertion (3) and (4).
\end{proof}

For $\kappa>0$, we set
\[
F_\kappa(x):=\tan x\cdot \cot(\kappa x)\quad(x\in\R).
\]
Let us investigate the behavior of $F_\kappa(x)$ on the interval $I_0=(0,\min(\frac{\pi}{\kappa},\pi))$.
Since $F_1\equiv 1$, we exclude the case $\kappa= 1$.
Notice that if $\kappa<2$, then $F_\kappa(x)$ has a pole at $x=\frac{\pi}{2}$ in the interval $I_0$.
At first, we see that
\[
F_\kappa(0):=
\lim_{x\to+0}F_\kappa(x)
=
\lim_{x\to+0}\frac{\sin x}{\sin(\kappa x)}\cdot\frac{\cos(\kappa x)}{\cos x}=\frac{1}{\kappa}>0.
\]
If $\kappa<1$, then it is obvious that $F_\kappa(\pi)=0$.
On the other hand,
if $1<\kappa<2$, then $\frac\pi2<\frac\pi\kappa<\pi$, 
and if $\kappa\ge 2$ then $0<\frac\pi\kappa\le \frac\pi2$,
and hence we have
\[
\lim_{x\to\frac\pi\kappa-0}F_\kappa(x)=
\begin{cases}
+\infty&(\text{if }1<\kappa<2),\\
-\infty&(\text{if }\kappa\ge 2).
\end{cases}
\]
By differentiating, we have
\[
F'_\kappa(x)=\frac{\cot(\kappa x)}{\cos^2x}-\frac{\kappa \tan x}{\sin^2(\kappa x)}
=
\frac{H_\kappa(2x)}{2(\cos x\sin(\kappa x))^2}.
\]
For the case $\frac12<\kappa<2$ ($\kappa\ne 1$),
Lemma \ref{lemma:func_H} tells us that $H_\kappa$ has a unique zero point $y_*$ in the interval $(0,\min(2\pi,\frac{2\pi}{\kappa}))$.
If we set $x_*=\frac{y_*}{2}$,
then we have $x_*\in I_0$ and $F_\kappa(x_*)=0$.

\begin{Lemma}
\label{lemma:func_F}
If $\frac12<\kappa<1$, then one has $F_\kappa(x_*)<1$,
and if $1<\kappa<2$, then one has $F_\kappa(x_*)>1$.
\end{Lemma}
\begin{proof}
We first assume that $\frac12<\kappa<1$.
In this case, since $0<1-\kappa<\frac12$ and since $\frac\pi2<x_*<\pi$, 
we have $\sin((1-\kappa)x_*)>0$.
Thus, since $\cos x_*<0$ and $\sin(\kappa x_*)>0$, we obtain
\[
\begin{array}{r@{\ }c@{\ }l}
\ds \sin((1-\kappa)x_*)>0
&\iff&
\ds \sin x_*\cos(\kappa x_*)-\cos x_*\sin(\kappa x_*)>0\\
&\iff&
\ds \sin x_*\cos(\kappa x_*)>\cos x_*\sin(\kappa x_*)\\
&\iff&
\ds F_\kappa(x_*)=\frac{\sin x_*\cos(\kappa x_*)}{\cos x_*\sin(\kappa x_*)}<1.
\end{array}
\]
Next, we assume that $1<\kappa<2$.
In this case, since $0<\kappa-1<1$ and since $\frac\pi2<x_*<\frac\pi\kappa<\pi$, 
we have $\sin((\kappa-1)x_*)>0$.
Thus, since $\cos x_*<0$ and $\sin(\kappa x_*)>0$, we obtain
\[
\begin{array}{r@{\ }c@{\ }l}
\ds \sin((\kappa-1)x_*)>0
&\iff&
\ds \cos x_*\sin(\kappa x_*)-\sin x_*\cos(\kappa x_*)>0\\
&\iff&
\ds \cos x_*\sin(\kappa x_*)>\sin x_*\cos(\kappa x_*)\\
&\iff&
\ds F_\kappa(x_*)=\frac{\sin x_*\cos(\kappa x_*)}{\cos x_*\sin(\kappa x_*)}>1.
\end{array}
\]
We have proved the lemma.
\end{proof}

Lemmas~\ref{lemma:func_H} and~\ref{lemma:func_F} yield the following table.

\begin{Lemma}
\label{lemma:table_F}
One has the following increasing/decreasing table of $F_\kappa$.
\[
\begin{array}{|l||c|l|}
\hline
\ds {\rm(A)}\ 0<\kappa\le\frac12&
\begin{array}{*{5}{c|}c}
x&0&\cdots&\frac{\pi}{2}&\cdots&\pi\\ \hline
F'_\kappa&&+&\times&\multicolumn{2}{c}{+}\\ \hline
F_\kappa&\frac1\kappa&\nearrow^{+\infty}&\times&{}_{-\infty}\nearrow&0
\end{array}
& \\ \hline\hline
\ds {\rm(B)}\ \frac12<\kappa<1&
\begin{array}{*{7}{c|}c}
x&0&\cdots&\frac{\pi}{2}&\cdots&x_*&\cdots&\pi\\ \hline
F'_\kappa&&+&\times&+&0&\multicolumn{2}{c}{-}\\ \hline
F_\kappa&\frac1\kappa&\nearrow^{+\infty}&\times&{}_{-\infty}\nearrow&F_\kappa(x_*)&\searrow&0
\end{array}
&F_\kappa(x_*)<1 \\ \hline\hline
\ds {\rm(C)}\ 1<\kappa<2&
\begin{array}{*{7}{c|}c}
x&0&\cdots&\frac{\pi}{2}&\cdots&x_*&\cdots&\frac{\pi}{\kappa}\\ \hline
F'_\kappa&&-&\times&-&0&+&\\ \hline
F_\kappa&\frac{1}{\kappa}&\searrow_{-\infty}&\times&{}^{+\infty}\searrow&F_\kappa(x_*)&\nearrow^{+\infty}&\times
\end{array}
&
F_\kappa(x_*)>1\\ \hline\hline
\ds {\rm(D)}\ \kappa\ge 2&
\begin{array}{*{3}{c|}c}
x&0&\cdots&\frac{\pi}{\kappa}\\ \hline
F'_\kappa&&-&\times\\ \hline
F_\kappa&\frac{1}{\kappa}&\searrow_{-\infty}&\times
\end{array}& \\ \hline
\end{array}
\]
\end{Lemma}

For $\kappa>0$, we set
\[
J_\kappa(x):=\frac{2x-2\kappa-1}{4\kappa x-2\kappa-1}=\frac{1}{2\kappa}\cdot
\frac{x-\frac{2\kappa+1}{2}}{x-\frac{2\kappa+1}{4\kappa}}
=
\frac{1}{2\kappa}-\frac{4\kappa^2-1}{8\kappa^2}\cdot\frac{1}{x-\frac{2\kappa+1}{4\kappa}}.
\]
Then, we have
\begin{equation}\label{eq:J-1}
J_\kappa(0)=1,\quad
J_{\kappa}(1)=-1,\quad
\lim_{x\to+\infty}J_\kappa(x)=\lim_{x\to-\infty}J_\kappa(x)=\frac{1}{2\kappa}.
\end{equation}
Note that, if we set $\kappa=\frac12$, then $J_{\frac12}\equiv 1$.
The following lemma is trivial.
\begin{Lemma}
\label{lemma:func_J}
Suppose that $\kappa\ne \frac12$.
Then, $J_\kappa$ has a pole at $x=\frac12+\frac{1}{4\kappa}$.
If $0<\kappa<\frac12$, then it is monotonic decreasing on $\R$,
and if $\kappa>\frac12$, then it is monotonic increasing on $\R$.
\end{Lemma}

We now consider the function $b(\theta)$.
If $\kappa=1$, then we have $b(\theta)=2(1-a)\cos\theta$.
Otherwise,
since $b(\theta)$ can be described as
\[
b(\theta)=\bigl((1-2a)+F_\kappa(\theta)\bigr)\cos\theta\quad(\text{if }\cos\theta\ne 0),
\]
the signature of $b(\theta)$ can be determined by using $F_\kappa$.
Note that, $\cos\theta=0$ occurs when $\kappa<2$, and in this case,
we have 
\[
b\Bigl(\frac\pi2\Bigr)=0+1\cdot\cot\frac{\kappa\pi}{2}=\cot\frac{\kappa\pi}{2}.
\]
It is positive if $\kappa<1$, and negative if $1<\kappa<2$.
These observations together with Lemma~\ref{lemma:func_F} yield the following table.

\begin{Lemma}
\label{lemma:signature b}
The signature of $b(\theta)$ on the interval $I_0=(0,\min(\pi,\frac\pi\kappa))$ is given as follows.
\[
\begin{array}{|l||c|l|} \hline
\multicolumn{3}{|c|}{0<\kappa\le\frac12}\\ \hline
    \ds 2a-1>\frac1\kappa
    &
    \begin{array}{*{5}{c|}c}
    \theta&0&\cdots&\varphi&\cdots&\pi\\ \hline 
    b(\theta)&&-&0&+&
    \end{array}
    &
    \ds\varphi<\frac{\pi}{2}\\ \hline
    \ds 0\le 2a-1\le \frac1\kappa&
    \begin{array}{*{3}{c|}c}
    \theta&0&\cdots&\pi \\ \hline
    b(\theta)&&+&
    \end{array}
    &\\ \hline
    \ds 2a-1<0&
    \begin{array}{*{5}{c|}c}
    \theta&0&\cdots&\varphi&\cdots&\pi\\ \hline
    b(\theta)&&+&0&-&
    \end{array}
    &
    \ds\varphi>\frac\pi2\\ \hline\hline
\multicolumn{3}{|c|}{\frac12<\kappa<1}\\ \hline
    \ds 2a-1>\frac1\kappa&
    \begin{array}{*{5}{c|}c}
    \theta&0&\cdots&\varphi&\cdots&\pi \\ \hline
    b(\theta)&&-&0&+&
    \end{array}&
    \ds\varphi<\frac\pi2\\ \hline
    \ds F_\kappa(\theta_*)<2a-1\le \frac1\kappa&
    \begin{array}{*{3}{c|}c}
    \theta&0&\cdots&\pi\\ \hline
    b(\theta)&&+&
    \end{array}&
    \\ \hline
    \ds 0\le 2a-1\le F_\kappa(\theta_*)&
    \begin{array}{*{7}{c|}c}
    \theta&0&\cdots&\varphi_1&\cdots&\varphi_2&\cdots&\pi\\ \hline
    b(\theta)&&+&0&-&0&+&
    \end{array}&
    \ds\varphi_i>\frac\pi2\\ \hline
    \ds 2a-1<0&
    \begin{array}{*{5}{c|}c}
    \theta&0&\cdots&\varphi&\cdots&\pi\\ \hline
    b(\theta)&&+&0&-&
    \end{array}&
    \ds\varphi>\frac\pi2\\ \hline \hline
\multicolumn{3}{|c|}{1<\kappa<2}\\ \hline
    \ds 2a-1\ge F_\kappa(\theta_*)&
    \begin{array}{*{7}{c|}c}
    \theta&0&\cdots&\varphi_1&\cdots&\varphi_2&\cdots&\frac\pi\kappa\\ \hline
    b(\theta)&&-&0&+&0&-&
    \end{array}&
    \ds\varphi_i>\frac\pi2\\ \hline
    \ds \frac1\kappa\le 2a-1<F_\kappa(\theta_*)&
    \begin{array}{*{3}{c|}c}
    \theta&0&\cdots&\frac\pi\kappa\\ \hline
    b(\theta)&&-&
    \end{array}
    &    \\ \hline
    \ds 2a-1<\frac1\kappa&
    \begin{array}{*{5}{c|}c}
    \theta&0&\cdots&\varphi&\cdots&\frac\pi\kappa\\ \hline
    b(\theta)&&+&0&-&
    \end{array}&
    \ds\varphi<\frac\pi2\\ \hline\hline
\multicolumn{3}{|c|}{\kappa\ge 2}\\ \hline
    \ds 2a-1\ge \frac1\kappa&
    \begin{array}{*{3}{c|}c}
    \theta&0&\cdots&\frac\pi\kappa\\ \hline
    b(\theta)&&-&
    \end{array}&
    \\ \hline
    \ds 2a-1<\frac1\kappa&
    \begin{array}{*{5}{c|}c}
    \theta&0&\cdots&\varphi&\cdots&\frac\pi\kappa\\ \hline
    b(\theta)&&+&0&-&
    \end{array}&
    \phantom{a}\\ \hline
\end{array}
\]
In this table,
$\varphi$ or $\varphi_i$ $(i=1,2)$ are solutions in $I_0$ of $b(\theta)=0$.
If $2a-1=F_\kappa(x_*)$, then $\varphi_1=\varphi_2$.
\end{Lemma}

Let us consider the function $b'(\theta)$.
If $\kappa=\frac12$, then we have
\begin{equation}
    \label{eq:b prime k12}
    b'(\theta)=(1-2a)\cos\theta+2\sin\theta\frac{\cos\frac\theta2}{\sin\frac\theta2}=
    (1-2a)\cos\theta+2\cos^2\Bigl(\frac\theta2\Bigr)
    =
    2(a-1)\sin\theta.
\end{equation}
Let us assume that $\kappa\ne\frac12$.
Since the function $b(\theta)$ can be also written as
\[
\begin{array}{r@{\ }c@{\ }l}
b(\theta)
&=&
\ds
\cos\theta+\sin\theta\cot(\kappa\theta)-2a\cos\theta
=
\frac{\cos\theta\sin(\kappa\theta)+\sin\theta\cos(\kappa\theta)}{\sin(\kappa\theta)}
-
2a\cos\theta\\
&=&
\ds
\frac{\sin((\kappa+1)\theta)}{\sin(\kappa\theta)}
-
2a\cos\theta,
\end{array}
\]
its derivative can be written by using $H_\alpha$ as
\begin{equation}
\label{eq:bp2}
    b'(\theta)=\frac{H_{2\kappa+1}(\theta)}{2\sin^2(\kappa\theta)}+2a\sin\theta.
\end{equation}
In fact,
\[
\begin{array}{r@{\ }c@{\ }l}
b'(\theta)
&=&
\ds
\frac{(\kappa{+1})\cos((\kappa+1)\theta)\sin\kappa\theta-\kappa\sin((\kappa{+1}))\theta\cos\kappa\theta}{\sin^2\kappa\theta}+2a\sin\theta\\[1em]
&=&
\ds
\frac{-\kappa\sin\theta +\cos((\kappa+1)\theta)\sin\kappa\theta}{\sin^2\kappa\theta}+2a\sin\theta\\[1em]
&=&
\ds
\frac{-\kappa\sin\theta+\frac12(\sin((2\kappa+1)\theta)-\sin\theta)}{\sin^2\kappa\theta}+2a\sin\theta\\[1em]
&=&
\ds
\frac{\sin((2\kappa+1)\theta)-(2\kappa+1)\sin\theta}{2\sin^2\kappa\theta}+2a\sin\theta\\
&=&
\ds
\frac{H_{2\kappa+1}(\theta)}{2\sin^2(\kappa\theta)}+2a\sin\theta.
\end{array}
\]
Let us set
\begin{equation}\label{def:B}
B(\theta):=2\sin^2(\kappa\theta)b'(\theta)=H_{2\kappa+1}(\theta)+4a\sin\theta\sin^2(\kappa\theta),\quad
\ell(\kappa,a):=4a\kappa-2\kappa-1.
\end{equation}
Then, the signatures of $b'(\theta)$ and $B(\theta)$ are the same,
and the derivative of $B(\theta)$ is given as,
if $\ell(\kappa,a)\ne0$ then
\begin{equation}
\label{eq:Bprime}
    B'(\theta)
    =
    2\ell(\kappa,a)\cos\theta\sin^2(\kappa\theta)\Bigl(F_\kappa(\theta)+J_\kappa(a)\Bigr),
\end{equation}
and if $\ell(\kappa,a)=0$ then
\begin{equation}\label{eq:Bprime2}
    B'(\theta)
    =
    \frac{1-4\kappa^2}{\kappa}\cos\theta\sin^2(\kappa\theta).
\end{equation}
In fact, we have
\[
\begin{array}{r@{\ }c@{\ }l}
B'(\theta)
&=&
\ds
-2(2\kappa+1)\sin((\kappa+1)\theta)\sin(\kappa\theta)
+
4a(\sin\kappa\theta)(\cos\theta\sin(\kappa\theta)+2\kappa\sin\theta\cos(\kappa\theta))\\
&=&
\ds
(4a-4\kappa-2)\cos\theta\sin^2(\kappa\theta)
+
(8a\kappa-4\kappa-2)\sin\theta\sin(\kappa\theta)\cos(\kappa\theta)\\
&=&
\ds
2\cos\theta\sin^2(\kappa\theta)(2a-2\kappa-1+(4a\kappa-2\kappa-1)\tan\theta\cot(\kappa\theta))\\
&=&
\ds
2(4a\kappa-2\kappa-1)\cos\theta\sin^2(\kappa\theta)\Bigl(F_\kappa(\theta)+\frac{2a-2\kappa-1}{4a\kappa-2\kappa-1}\Bigr).
\end{array}
\]
Note that two conditions $\ell(\kappa,a)=0$ and $2a-2\kappa-1=0$ occur simultaneously only when $\kappa=\frac12$.
Set
\[
G(x):=\frac{2x^2+3x+1}{6x}.
\]

\begin{Lemma}
\label{lemma:func_b prime}
\begin{enumerate}[\rm(1)]
    \item Suppose that $0<\kappa<\frac12$, or $\kappa>1$.
    If $a> G(\kappa)$, then there exists a unique $\varphi_*\in I_0$ such that $b'(\varphi_*)=0$,
    and one has $b'(\theta)>0$ for $\theta\in(0,\varphi_*)$, and $b'(\theta)<0$ for $\theta\in(\varphi_*,\pi)$.
    If $a\le G(\kappa)$, then one has $b'(\theta)$ for any $\theta\in I_0$.
    \item Suppose that $\frac12<\kappa<1$.
    If $a<G(\kappa)$, then there exists a unique $\varphi_*\in I_0$ such that $b'(\varphi_*)=0$,
    and one has $b'(\theta)<0$ for $\theta\in(0,\varphi_*)$, and $b'(\theta)>0$ for $\theta\in(\varphi_*,\pi)$.
    If $a<G(\kappa)$, then one has $b'(\theta)>0$ for any $\theta\in I_0$.
\end{enumerate}

\end{Lemma}
\begin{proof}
Since we have \eqref{eq:Bprime} or \eqref{eq:Bprime2}, we can use Lemmas \ref{lemma:func_F} and \ref{lemma:func_J}.
This lemma can be obtained by dividing cases and by elementary but tedious calculations.
Note that $\frac{2\kappa^2+3\kappa+1}{6\kappa}$ comes from the equation $0=F_{\kappa}(0)+J_\kappa(a)$.
By differentiating $b(\theta)$ by using expression \eqref{eq:b},
we obtain
\[
b'(\theta)=(2a-1)\sin\theta+\cos\theta\cot(\kappa\theta)-\frac{\kappa\sin\theta}{\sin^2(\kappa\theta)},
\]
and hence
\[
\begin{array}{r@{\ }c@{\ }l}
\ds
b'(\theta)\sin^2(\kappa\theta)
&=&
\ds
(2a-1)\sin\theta(1-\cos^2(\kappa\theta))+\cos\theta\sin(\kappa\theta)\cos(\kappa\theta)-\kappa\sin\theta\\
&=&
\ds
(2a-\kappa-1)\sin\theta+(2a-1)\sin\theta\cos^2(\kappa\theta)+\cos\theta\sin(\kappa\theta)\cos(\kappa\theta)\\
&=&
\ds
(2a-\kappa-1)\sin\theta+
\cos\theta\sin(\kappa\theta)\cos(\kappa\theta)((2a-1)F_\kappa(\theta)+1).
\end{array}
\]
If $\varphi_i$ satisfies $B'(\varphi_i)=0$, then we have $F_\kappa(\varphi_i)+J_\kappa(a)=0$ and hence
\begin{equation}\label{eq:lemma18-1}
(2a-1)F_\kappa(\varphi_i)+1
=
\frac{-(2a-1)(2a-2\kappa-1)+(4a\kappa-2\kappa-1)}{\ell(\kappa,a)}
=
\frac{4a(a-1)}{\ell(\kappa,a)}.
\end{equation}

The non-trivial cases are given as
(i) $\frac12<\kappa<1$ and $0\le -J_\kappa(a)\le F_\kappa(x_*)$,
or 
(ii) $1<\kappa<2$ and $-J_\kappa(x)>F_\kappa(x_*)$.

first we consider the case (i).
In this case,
Lemma~\ref{lemma:func_J} tells us that we have $\ell(\kappa,a)>0$.
Since $-J_\kappa(a)$ is monotonic decreasing when $\ell(\kappa,a)>0$ and since $-J_\kappa(1)=1$ by \eqref{eq:J-1},
we see that $a$ need satisfy $a\in(1,\frac12+\kappa)$
because $F_\kappa(x_*)<1$ by Lemma~\ref{lemma:func_F}.
Thus, we have $2a-\kappa-1>0$.
Again by Lemma~\ref{lemma:func_F},
the equation $B'(\theta)=F_\kappa(\theta)+J_\kappa(a)=0$ has
at most two solutions $\varphi_i\in(\frac{\pi}{2\kappa},\pi)$, $i=1,2$.
Set $\varphi_1\le \varphi_2$.
Then, since 
\[
B(0)=0,\quad
B(\pi)=\sin(2\kappa\pi)>0,
\]
we have
\[
\begin{array}{*{9}{c|}c}
\theta&0&\cdots&\frac\pi2&\cdots&\varphi_1&\cdots&\varphi_2&\cdots&\pi\\ \hline
\ell(\kappa,a)&&+&&+&&+&&+\\ \hline
\cos\theta&&+&0&-&&-&&-&\\ \hline
F_\kappa+J_\kappa&&+&&-&0&+&0&-&\\ \hline \hline
B'(\theta)&&+&&+&0&-&0&+&\\ \hline
B(\theta)&0&\nearrow&&\nearrow&&\searrow&&\nearrow&B(\pi)>0
\end{array}
\]
We shall show that $B(\varphi_i)>0$, $(i=1,2)$,
which implies $b'(\theta)>0$.
Since $a>1$, we have $(2a-1)F_\kappa(\varphi_i)+1>0$ by \eqref{eq:lemma18-1}.
Since $\varphi_i\in(\frac{\pi}{2\kappa},\pi)$, we see that $\cos\varphi_i\sin(\kappa\varphi_i)\cos(\kappa\varphi_i)>0$ and hence we obtain $B(\varphi_i)>0$.

Next, we consider the case (ii).
Then, since $F_\kappa(x_*)>1$ by Lemma~\ref{lemma:func_F} and since $J_\kappa(1)=-1$ by \eqref{eq:J-1}, 
we need to have $\ell(\kappa,a)>0$ and $a<1$,
whence $2a-\kappa-1<0$.
Lemma~\ref{lemma:func_J} tells us that the function $-J_\kappa(a)$ is monotonic decreasing if $\ell(\kappa,a)>0$.
Again by Lemma~\ref{lemma:func_F},
the equation $B'(\theta)=F_\kappa(\theta)+J_\kappa(a)=0$ has
at most two solutions $\varphi_i\in(\frac{\pi}{2\kappa},\pi)$, $i=1,2$.
Let $\varphi_1\le \varphi_2$.
Then, since
\[
B(0)=0,\quad
B\Bigl(\frac\pi\kappa\Bigr)=-2\kappa\sin\frac\pi\kappa<0,
\]
we have
\[
\begin{array}{*{9}{c|}c}
\theta&0&\cdots&\frac\pi2&\cdots&\varphi_1&\cdots&\varphi_2&\cdots&\frac\pi\kappa\\ \hline
\ell(\kappa,a)&&+&&+&&+&&+\\ \hline
\cos\theta&&+&0&-&&-&&-&\\ \hline
F_\kappa+J_\kappa&&-&&+&0&-&0&+&\\ \hline \hline
B'(\theta)&&-&&-&0&+&0&-&\\ \hline
B(\theta)&0&\searrow&&\searrow&&\nearrow&&\searrow&B(\frac\pi\kappa)<0
\end{array}
\]
Since $a<1$, we have $(2a-1)F_\kappa(\varphi_i)+1<0$ b y\eqref{eq:lemma18-1}.
Since $\varphi_i\in(\frac\pi2,\frac\pi\kappa)$, we have $\cos\varphi_i\sin(\kappa\varphi_i)\cos(\kappa\varphi_i)>0$
so that we obtain $B(\varphi_i)<0$.
\end{proof}

\section{The domain $\Omega$ for $\kappa<+\infty$}

In this section,
we shall determine $\Omega$ for finite $\kappa>0$, which is the connected component of the set $\set{z=x+yi\in\domainF}{F(x,y)>0}$ containing $z=0$.
Let $r_\pm(\theta)$ be the solutions of the equation \eqref{eq:r}
and let $\alpha_i$, $i=1,2$ are the solutions of the equation \eqref{eq:quad}.
Since $F(x,y)$ is a continuous function, 
the boundary $\partial\Omega$ is included in the set $\set{z=x+yi\in\C}{F(x,y)=0}\subset f_{\kappa,\gamma}^{-1}(\R)$.
Since $F$ is even function on $y$, the domain $\Omega$ is symmetric with respect to the real axis,
and hence we consider $D:=\Omega\cap\C^+$.

\begin{Proposition}
\label{prop:Omega k=1}
Suppose that $\kappa=1$.
\begin{enumerate}[\rm(1)]
    \item If $\gamma>1$, then one has $\Omega=\C\setminus\{-\frac1\gamma\}$.
    \item If $0<\gamma\le 1$, then one has 
    $\Omega=\set{z=x+yi\in\domainF}{\bigl(x+\frac1\gamma\bigr)^2+y^2> \frac{1-\gamma}{\gamma^2}}$.
    \item If $\gamma=0$, then one has $\Omega=\set{z=x+yi\in\domainF}{1+2x>0}$.
    \item If $\gamma<0$, then one has 
    $\Omega=\set{z=x+yi\in\domainF}{\bigl(x+\frac1\gamma\bigr)^2+y^2< \frac{1-\gamma}{\gamma^2}}$.
    In particular, $\Omega$ is bounded.
\end{enumerate}
\end{Proposition}

\begin{proof}
For $z=x+yi$, we have
\[
f_{1,\gamma}(z)=\frac{(x+\gamma x^2+\gamma y^2+yi)(x+1+yi)}{(1+\gamma x)^2+\gamma^2y^2}=
\frac{1}{{(1+\gamma x)^2+\gamma^2y^2}}\Bigl(\begin{array}{l}
(x+\gamma x^2+\gamma y^2)(1+x)-y^2\\
\quad iy(1+2x+\gamma x^2+\gamma y^2)
\end{array}\Bigr).
\]
Thus, $f_{1,\gamma}(z)\in\R$ implies $y=0$ or 
\[
1+2x+\gamma x^2+\gamma y^2=0\iff
\begin{cases}
\Bigl(x+\frac{1}{2\gamma}\Bigr)^2+y^2=\frac{1-\gamma}{\gamma^2}&(\text{if }\gamma\ne 0),\\
x=-\frac12&(\text{if }\gamma=0).
\end{cases}
\]
Therefore, we have $\Omega=\set{z\in\domainF}{\mathrm{Re}\,z>-\frac12}$ when $\gamma=0$ because it contains $z=0$.
Suppose that $\gamma\ne 0$ and set $C=\set{z=x+yi\in\C}{1+2x+\gamma x^2+\gamma y^2=0}$.
If $\gamma >1$, then $C=\emptyset$, which implies $\Omega=\C\setminus\{-\frac1\gamma\}$.
Assume that $0<\gamma\le 1$.
Then, $C\ne\emptyset$ and $z=0$ does not contained in the interior of $C$ and hence
$\Omega$ is the out side of $C$, which is written as
$\Omega=\set{z=x+yi\in\C}{1+2x+\gamma x^2+\gamma y^2>0}$.
Assume that $\gamma<0$.
Then, $C\ne\emptyset$ and $z=0$ is contained in the interior of $C$,
and thus $\Omega$ is the interior of $C$, which is written as 
\[\Omega=\set{z=x+yi\in\domainF}{1+2x+\gamma x^2+\gamma y^2>0}.\]
The proof is completed.
\end{proof}

Recall that $a=\kappa\gamma$.
Set $\theta_0:=\frac{\pi}{\kappa}$ and let $I_0$ be the interval $(0,\min(\pi,\theta_0))$.

\begin{Proposition}
\label{prop:Omega}
Let $\kappa>0$ with $\kappa\ne 1$.
For $z\in\C\setminus\{x\le -\kappa\}$, one sets $re^{\theta}=1+\frac{z}{\kappa}$.
Then, $\Omega$ can be described as follows.
\begin{enumerate}[\quad\rm(1)]
    \item If $a<0$, then there exists a unique $\theta_*\in(0,\frac{\pi}{\kappa+1})$ such that $r_+(\theta_*)=r_-(\theta_-)$ and
    that $r_\pm(\theta)$ are both positive with $r_-(\theta)\le r_+(\theta)$ on $(0,\theta_*)$.
    Moreover, one has
    \[
    \Omega=\set{z\in\domainF}{|\theta|<\theta_*\text{ and } r_-(\theta)<r<r_+(\theta)}.
    \]
    In particular, $\Omega$ is bounded.
    One has $\alpha_1,\alpha_2\in\partial\Omega$ and $-\frac1\gamma\in\overline{\Omega}$, whereas $-\kappa\not\in\overline{\Omega}$.
    \item If $a=0$, then one has $r(\theta)=r_\pm(\theta)=\frac{\sin(\kappa\theta)}{\sin((\kappa+1)\theta)}$ which is positive on the interval $(0,\frac{\pi}{\kappa+1})$.
    Moreover, one has
    \[
    \Omega=\set{z\in\domainF}{|\theta|<\frac{\pi}{\kappa+1}\text{ and }r>r(\theta)}.
    \]
    $\Omega$ has an asymptotic line $y=\pm(x+\frac{\kappa^2}{\kappa+1})\tan \theta_1$.
    One has $\alpha_1=\alpha_2=-\frac{\kappa}{\kappa+1}\in\partial\Omega$ and $-\kappa\not\in\overline{\Omega}$.
    \item If $0<a<1$, then $r_+(\theta)$ is the only positive solution of $\eqref{eq:r}$ on $I_0$,
    and one has
    \[
    \Omega=\set{z\in\domainF}{|\theta|<\min(\theta_0,\pi)\text{ and } r>r_+(\theta)}.
    \]
    If $\kappa>1$, then $\Omega$ has an asymptotic line $y=\pm(\tan\theta_0)(x+\kappa-\frac1a)$.
    One has $\alpha_2\in\partial\Omega$, whereas $\alpha_1,-\frac1\gamma,-\kappa\not\in\overline{\Omega}$.
    \item Suppose that $a=1$ and $\kappa\ne 1$.
        \begin{enumerate}[\rm(a)]
            \item If $0<\kappa<1$, then one has $r_+(\theta)=0$ and $r_-(\theta)=-b(\theta)<0$ for $\theta\in I_0$, and
            $\Omega=\domainF$.
            One has $\alpha_1\in\partial\Omega$ and $\alpha_2=-\kappa=-\frac1\gamma\in\partial\Omega$.
            \item If $\kappa>1$, then one has $r_+(\theta)=-b(\theta)>0$ and $r_-(\theta)=0$ for $\theta\in I_0$,
            and one has
            \[
            \Omega=\set{z\in\domainF}{|\theta|<\theta_0\text{ and }r>r_+(\theta)}.
            \] 
            One has $\alpha_2=-1\in\partial\Omega$, while $\alpha_1=-\kappa=-\frac1\gamma\not\in\overline{\Omega}$.
            Moreover, $\Omega$ has an asymptotic line $y=\pm(\tan\theta_0)(x+\kappa-\frac1a)$.
        \end{enumerate}
    \item Suppose that $a>1$.
        \begin{enumerate}[\rm(a)]
            \item If $\kappa>1$ with $D(0)\ge 0$, 
            then $r_\pm(\theta)$ are both positive in $I_0$ 
            with $r_-(\theta)\le r_+(\theta)$, and one has
        \[
        \Omega=\set{z\in\domainF}{|\theta|<\theta_0\text{ and }r>r_+(\theta)}.
        \]
        $\Omega$ has an asymptotic line $y=\pm(\tan\theta_0)(x+\kappa-\frac1a)$.
        One has $\alpha_2\in\partial\Omega$, while $-\kappa,-\frac1\gamma,\alpha_1\not\in\overline{\Omega}$.
        \item If $\kappa>1$ and $D(0)<0$, 
        then there exists a unique $\theta_*\in(0,\theta_0)$ such that $D(\theta_*)=0$, 
        and $r_\pm(\theta)$ are both positive in the interval $(\theta_*,\theta_0)$.
        Moreover, one has
        \[
        \Omega=\set{z\in\domainF}{|\theta|<\theta_0\text{ and if }|\theta|\ge \theta_*\text{ then }0<r<r_-(\theta)\text{ or }r>r_+(\theta)}.
        \]
        In this case, $\alpha_i$, $i=1,2$ are both non-real and one has $\alpha_i,-\kappa\in\partial\Omega$ and $-\frac1\gamma\in\overline{\Omega}$.
        Moreover,
        $\Omega$ has an asymptotic line $y=\pm(\tan\theta_0)(x+\kappa-\frac1a)$.
        \item If $0<\kappa< 1$, 
        then there are no $\theta$ such that $D(\theta)>0$, and one has $\Omega=\domainF$.
        \end{enumerate}
\end{enumerate}
\end{Proposition}
Recall that, since $\Omega$ is symmetric with respect to the real axis,
it is enough to determine the boundary $\partial\Omega$ of $\Omega$ in the upper half plane.

\begin{proof}
(1) Assume that $a<0$.
Since $r_+(\theta)\cdot r_-(\theta)=\frac{a-1}{a}>0$,
we see that $r_\pm(\theta)$ have the same signature if $r_\pm(\theta)\in\R$.
Since $a<0$, we have $b(0)=1+\frac1\kappa-2a>0$ by \eqref{def:b zero}.
Let $\theta_1:=\frac{\pi}{\kappa+1}$.
Then, 
Lemmas~\ref{lemma:table of b prime} and \ref{lemma:func_H} tells us that
$b'(\theta)\le 0$ for any $\theta\in(0,\theta_1)$.
In fact, if $\kappa$ does not satisfy $\frac12<\kappa<1$, then it is a direct consequence of Lemma~\ref{lemma:table of b prime},
and if $\frac12<\kappa<1$ then we can verify it by Lemma~\ref{lemma:table of b prime} and by the fact that $b'(\theta_1)=H_{2\kappa+1}(\theta_1)+2a\sin\theta_1<0$ by Lemma~\ref{lemma:func_H}.
Therefore, $b(\theta)$ is monotonic decreasing on the interval $(0,\theta_1)$.
If $b(\theta_1)<0$, then there exists a unique $\varphi\in(0,\theta_1)$ such that $b(\varphi)=0$,
and if $b(\theta_1)\ge 0$ then we set $\varphi=\theta_1$.
Since we have
\[
D(0)=\Bigl(1+\frac1\kappa\Bigr)^2-\frac{4a}{\kappa}>0
\quad\text{and}\quad
\begin{array}{r@{\ }c@{\ }l}
D(\theta_1)
&=&
\ds
(-2a\cos\theta_1)^2-4a(a-1)
=
4a^2\cos^2\theta_1-4a^2+4a\\
&=&
4a-4a^2\sin^2\theta_1
=
4a{(1-a\sin^2\theta_1)}<0
\end{array}
\]
and $D'(\theta)=2b(\theta)b'(\theta)$,
the increasing/decreasing table of $D(\theta)$ is given as
\[
\begin{array}{*{5}{c|}c}
\theta&0&\cdots&\varphi&\cdots&\theta_1\\ \hline
b(\theta)&&+&0&-\\ \hline
b'(\theta)&&-&-&-\\ \hline
D'(\theta)&&-&0&+\\ \hline
D(\theta)&D(0)>0&\searrow&D(\varphi)&\nearrow&D(\theta_1)<0
\end{array}
\]
and hence there exists a unique $\theta_*\in(0,\varphi)$ such that $D(\theta_*)=0$.
In particular, $D$ is monotonic decreasing in the interval $(0,\theta_*)$,
and $D(\theta_*+\delta)<0$ for $\delta>0$ such that $\theta_*+\delta<\varphi$.
Therefore,
since $r_+(\theta)+r_-(\theta)=-b(\theta)/a$,
the signatures of $r_\pm(\theta)$ is the same as that of $b(\theta)$ if $r_\pm$ are real
so that
$r_{\pm}$ are positive on $(0,\theta_*]$ and
$r_{\pm}$ are not real for $\theta\in(\theta_*,\varphi)$.
Since $r_+(\theta_*)=r_-(\theta_*)$ by the fact $D(\theta_*)=0$,
the curves 
$r_+(\theta)$, $\theta\in (0, \theta_*]$
followed by 
$r_-(\theta_*-\theta)$,  $\theta\in (0, \theta_*]$, 
form a continuous curve going from $\alpha_2$ to $\alpha_1$ in the upper half-plane.
Let $\alpha_i$, $i=1,2$ be the solutions of \eqref{eq:quad}.
Since $\alpha_1\alpha_2=\frac1\gamma<0$ and $\alpha_1\le \alpha_2$,
we have $\alpha_1<0<\alpha_2$ so that
\[
\Omega=\set{z\in\domainF}{|\theta|<\theta_*\text{ and } r_-(\theta)<r<r_+(\theta)}.
\]
Since $q(-\frac1\gamma)>0$ and $q(-\kappa)<0$ by \eqref{eq:q special values}, 
where $q$ is defined in \eqref{def:b zero},
we see that $-\frac1\gamma\in\overline{\Omega}$ and $-\kappa\not\in\overline{\Omega}$.


\noindent
(2) Assume that $a=0$.
In this case,
the equation \eqref{eq:r} reduces to $b(\theta)r-1=0$ so that 
\[r_\pm(\theta)=r(\theta)=\frac{1}{b(\theta)}=\frac{1}{\cos\theta+\sin\theta\cot(\kappa\theta)}=\frac{\sin(\kappa\theta)}{\sin((\kappa+1)\theta)}.\]
Let $\theta_1=\frac{\pi}{\kappa+1}$.
Since $\sin(\kappa\theta)$ and $\sin((\kappa+1)\theta)$ are both positive in the interval $(0,\theta_1)$,
and since $\ds\lim_{\theta\to\theta_1-0}\sin((\kappa+1)\theta)=0$,
we see that
\[
\ds\lim_{\theta\to\theta_1-0}r(\theta)=+\infty.
\]
Thus, 
it has an asymptotic line with slope $\tan\theta_1$, which is determined later.
Since $\gamma=a/\kappa=0$,
the solutions $\alpha_i$, $i=1,2$ are given as $\alpha_1=\alpha_2=-\frac{\kappa}{\kappa+1}$.
Since $q(0)=1>0$ and $q(-\kappa)=-\kappa<0$,
we see that the domain $\Omega$ is given as
\[
\Omega=\set{z=x+yi\in\domainF}{|\theta|<\theta_1,\ r>r(\theta)}
\]
and $-\kappa\not\in\overline{\Omega}$.


\noindent
(3) Assume that $0<a<1$.
In this case,
we have
\[
D(\theta)=b(\theta)^2+4a(1-a)>0
\]
for any $\theta\in I_0$
so that the solutions of \eqref{eq:r} are always real.
Moreover,
since $r_+(\theta)\cdot r_-(\theta)=-\frac{1-a}{a}<0$,
they have the different signatures,
and 
$r_+(\theta)$ is the positive real solution of \eqref{eq:r} by 
the fact $|b(\theta)|<\sqrt{D(\theta)}$.
In particular, $r_+(\theta)$, $\theta\in I_0$ forms a continuous curve in $\C^+$.
The solutions $\alpha_i$, $i=1,2$ of \eqref{eq:quad} are both negative
because we have $\gamma=a/\kappa>0$ and $\alpha_1+\alpha_2=-(1+\frac1\kappa)/\gamma<0$.
Since $q(-\frac1\gamma)<0$ and $q(-\kappa)<0$ by \eqref{eq:q special values}
and since $\frac1\gamma>\kappa$, 
we see that 
\begin{equation}\label{eq:prop22-1}
\alpha_1<-\frac1\gamma<-\kappa<\alpha_2<0.
\end{equation}

\noindent(a)
We first assume that $\kappa\ge 1$.
Then, $b(\theta)$ is defined on the interval $I_0=(0,\theta_0)$,
and we have
$\lim\limits_{\theta\to\theta_0-0}b(\theta)=-\infty$,
which implies 
\[\lim_{\theta\to\theta_0-0}r_+(\theta)=+\infty.\]
Therefore,
the curve $r_+(\theta)$, $\theta\in(0,\theta_0)$ has the asymptotic line with gradient $\tan\theta_0$,
which is determined later.
By \eqref{eq:prop22-1},
we see that $\Omega$ is given as
\[\Omega=\set{z\in\domainF}{|\theta|<\theta_0,\ r>r_+(\theta)},\]
and $\alpha_1,-\frac1\gamma,-\kappa\not\in\overline{\Omega}$.

\noindent(b)
Next, we assume that $0<\kappa<1$.
Then, $b(\theta)$ is defined for any $\theta\in I_0=(0,\pi)$.
If $\theta=0$, then the positive solution of \eqref{eq:r} corresponds to $\alpha_2<0$,
and if $\theta=\pi$ then, 
since $b(\pi)=2a-1$, 
the positive solution of \eqref{eq:r} corresponds to $-\frac1\gamma$.
Thus,
the curve $r_+(\theta)$, $\theta\in(0,\pi)$ connects $z=\alpha_2$ and $z=-\frac1\gamma$ passing in the upper half plane.
This means that $z=-\kappa$ is in the interior of the curve,
whereas $z=0$ is in its outside.
Since $\Omega$ is the connected component including $z=0$,
$\Omega$ is the outside of the curve and hence
\[
\Omega=\set{z\in\domainF}{|\theta|<\pi,\ r>r_+(\theta)}.
\]


\noindent
(4) Assume that $a=1$ and $\kappa\ne 1$.
In this case, the equation \eqref{eq:r} reduces to $r^2+b(\theta)r=0$, whose solutions are
$r(\theta)=0$, $-b(\theta)$.
Recall that $I_0=(0,\min(\pi,\frac\pi\kappa))$.
By Lemmas~\ref{lemma:signature b} and \ref{lemma:func_F},
we see that if $0<\kappa<1$ then $b(\theta)>0$ and if $\kappa>1$ then $b(\theta)<0$ for any $\theta\in I_0$.
This means that if $0<\kappa<1$ then the equation \eqref{eq:r} does not have a positive solution,
and hence we have 
\[
\Omega=\set{z\in\domainF}{|\theta|<\pi,\ r>0}=\domainF,
\]
which shows the assertion (4)-(a).
If $\kappa>1$, then since $b'(\theta)<0$ for $\theta\in I_0$ and 
since $b(0)=1+\frac1\kappa-2<0$ and $\lim_{\theta\to\theta_0-0}b(\theta)=-\infty$,
the function $D(\theta)$ is monotonic increasing on $I_0$ and hence
we have 
\[
\lim_{\theta\to\theta_0-0}r_+(\theta)=+\infty.
\]
Thus, the curve $r_+(\theta)$, $\theta\in I_0$ has an asymptotic line with gradient $\tan\theta_0$,
which is determined later.
Thus, we have
\[
\Omega=\set{z\in\domainF}{|\theta|<\theta_0\text{ and } r>r_+(\theta)}.
\]


\noindent
(5) Suppose that $a>1$.
In this case, we have 
$b(0)=1+\frac1\kappa-2a$ and
\[
b\Bigl(\frac\pi\kappa\Bigr)=(1-2a)\cos\frac\pi\kappa\quad(\text{if }\kappa>1),
\quad
b(\pi)=2a-1>0\quad(0<\kappa<1).
\]
Note that $b(0)>0$ if $\kappa>1$.
Since $r_+(\theta)\cdot r_-(\theta)=\frac{a-1}{a}>0$,
two solutions $r_\pm(\theta)$ of \eqref{eq:r} have the same signature if $r_\pm(\theta)$ are real.

\noindent(a)
We first consider the case $\kappa>1$ and $D(0)\ge 0$.
Let us show that $D(\theta)>0$ for $\theta\in I_0$.
Set
\[
K(x):=\frac{x}{4}\Bigl(1+\frac1x\Bigr)^2,\quad
G(x)=\frac{2x^2+3x+1}{6x}\quad(x>0).
\]
Then, the condition $D(0)\ge 0$ is equivalent to $a\le K(\kappa)$,
and hence we have $a\le G(\kappa)$ because
\[
G(x)-K(x)
=
\frac{x^2-1}{12x}>0\quad\text{if}\quad x>1.
\]
By the assumption $\kappa>1$,
we see that $b'(\theta)<0$ for any $\theta\in I_0$ by Lemma \ref{lemma:table of b prime}
so that
$b$ is monotonic decreasing in this interval.
Since $b(0)<0$, the function $b$ is negative in $I_0$,
and hence
$D'(\theta)=2b(\theta)b'(\theta)>0$ so that $D(\theta)$ is monotonic increasing in the interval $I_0$,
and in particular,
it is positive on $I_0$.
By \eqref{eq:r prime},
we have $r'_+(\theta)>0$,
and hence the function $r_+(\theta)$ is monotonic increasing,
whereas $r_-(\theta)$ is monotonic decreasing because $r_-(\theta)=\frac{a-1}{ar_+(\theta)}$.
Note that
\[
\lim_{\theta\to\theta_0-0}r_+(\theta)=+\infty,\quad
\lim_{\theta\to\theta_0-0}r_0(\theta)=0.
\]
Thus, $r_+(\theta)$, $\theta\in I_0$ draws an bounded curve connecting $z=\alpha_2$ to $\infty$,
and $r_-(\theta)$, $\theta \in I_0$ draws a bounded curve connecting $z=\alpha_1$ to $z=-\kappa$.
Since we have $-\kappa<-\frac1\gamma<\alpha_1<\alpha_2<0$
and since $\Omega$ is the connected component including $z=0$,
we have
\[
\Omega=\set{z\in\domainF}{|\theta|<\theta_0,\ r>r_+(\theta)}.
\]


\noindent
(b) Next, we assume that $\kappa>1$ and $D(0)<0$.
According to Lemma \ref{lemma:func_b prime} (1),
we consider the function $b'(\theta)$ in two cases,
that is,
(i) $a\le G(\kappa)$ and (ii) $a> G(\kappa)$.

(i) Assume that $a\le G(\kappa)$.
Then, $b(\theta)$ is monotonic decreasing.
Since $b(0)<0$, we see that $b(\theta)<0$ for any $\theta\in I_0$
and therefore 
$D'(\theta)=2b(\theta)b'(\theta)>0$ for any $\theta\in I_0$.
Thus, $D(\theta)$ is monotonic increasing with $D(0)<0$.
In particular, there exists a unique $\theta_*$ such that $D(\theta_*)=0$,
and $r_\pm(\theta)$ are real for $\theta\in(\theta_*,\theta_0)$.
In this interval,
since $r_+(\theta)+r_-(\theta)=-b(\theta)/a>0$,
we see that $r_\pm(\theta)$ are both positive.
By \eqref{eq:r prime}, we have $r'_+(\theta)>0$ and thus
the function $r_+(\theta)$ is monotonic increasing,
whereas $r_-(\theta)$ is monotonic decreasing because $r_-(\theta)=\frac{a-1}{ar_+(\theta)}$.
By taking a limit $\theta\to\theta_0-0$, we have
\[
\lim_{\theta\to\theta_0-0}r_+(\theta)=+\infty,\quad
\lim_{\theta\to\theta_0-0}r_-(\theta)=0.
\]
This means that $r_+(\theta)$ draws an unbounded curve connecting $z=\alpha_1$ and $z=\infty$,
and $r_-(\theta)$ draws a bounded curve connecting $z=\alpha_1$ and $z=-\kappa$,
where $\alpha_1$ is the complex solution of \eqref{eq:quad} with positive imaginary part.
Since we have  $-\kappa<-\frac1\gamma<0$, 
the domain $\Omega$ is given as
\[
\Omega=\set{z\in\domainF}{|\theta|<\theta_0,\ \text{if }|\theta|\ge\theta_*\text{ then }0<r<r_-(\theta)\text{ or }r>r_+(\theta)}.
\]

(ii) Assume that $a>G(\kappa)$.
Then, there are two possibilities on $b(\varphi_*)$.
If $b(\varphi_*)\le 0$, then we have
\[
\begin{array}{*{5}{c|}c}
\theta&0&\cdots&\varphi_*&\cdots&\frac\pi\kappa\\ \hline
b'(\theta)&&+&0&-&\\ \hline
b(\theta)&-&-&&-&\\ \hline \hline
D'(\theta)&&-&0&+&\\ \hline
D(\theta)&D(0)<0&\searrow&&\nearrow^{+\infty}&\times
\end{array}
\]
and if $b(\varphi_*)>0$,
then there exist exactly two $\varphi_1<\varphi_2$ such that $b(\varphi_i)=0$ so that
\[
\begin{array}{*{9}{c|}c}
\theta&0&\cdots&\varphi_1&\cdots&\varphi_*&\cdots&\varphi_2&\cdots&\frac\pi\kappa\\ \hline
b'(\theta)&&\multicolumn{3}{c|}{+}&0&\multicolumn{3}{c|}{-}&\\ \hline
b(\theta)&-&-&0&\multicolumn{3}{c|}{+}&0&-&\\ \hline \hline
D'(\theta)&&-&0&+&0&-&0&+\\ \hline
D(\theta)&D(0)<0&\searrow&&\nearrow&&\searrow&&\nearrow^{+\infty}&\times
\end{array}
\]
We note that $D(\varphi_i)=b(\varphi_i)^2-4a(a-1)<0$.
Since $r_+(\theta)+r_-(\theta)=-b(\theta)/a$,
if $D(\varphi_*)>0$, then $r_\pm(\theta)$ are both negative so that we do not deal with this case.
Thus, in both cases, there exists a unique $\theta_*\in I_0$ such that
$D(\theta_*)=0$ and $r_\pm(\theta)>0$ for any $\theta\in(\theta_*,\theta_0)$.
By \eqref{eq:r prime},
we see that $r_+(\theta)$ is monotonic increasing on $(\theta_*,\theta_0)$,
whereas $r_-(\theta)$ is monotonic decreasing.
Moreover,
we have
\[
\lim_{\theta\to\theta_0-0}r_+(\theta)=+\infty,\quad
\lim_{\theta\to\theta_0-0}r_-(\theta)=0,
\]
and hence the curves $r_\pm(\theta)$, $\theta\in(\theta_*,\theta_0)$ form an unbounded curve connecting $z=\alpha_1$ and $z=\infty$.
Since $-\kappa<-\frac1\gamma<0$, we have
\[
\Omega=\set{z\in\domainF}{|\theta|<\theta_0,\ \text{and if }|\theta|\ge\theta_*\text{ then }0<r<r_-(\theta)\text{ or }r>r_+(\theta)}.
\]


\noindent
(c) We finally assume that $0<\kappa<1$.
In this case, we have
$D(\pi)=(2a-1)^2-4a(a-1)=1$.
We note that $b(0)<0$ implies $D(0)<0$.
In fact, $b(0)<0$ means $1+\frac1\kappa<2a$ so that 
\begin{equation}\label{ineq:prop2-2}
D(0)=\Bigl(1+\frac1\kappa\Bigr)^2-\frac{4a}{\kappa}
<
2a\Bigl(1+\frac1\kappa\Bigr)-\frac{4a}{\kappa}
=
2a\Bigl(1-\frac1\kappa\Bigr)<0.
\end{equation}
Since $a>1$, the signatures of $r_\pm(\theta)$ are the same,
and they are the opposite to the signature of $b(\theta)$.

Let $I_D\subset I_0$ be the maximal interval such that $D$ is positive on $I_D$.
We shall show that there are no suitable solutions of \eqref{eq:r}, that is,
$r_\pm(\theta)<0$ for any $\theta\in I_D$,
which yields that $\Omega=\domainF$.
Let us recall Lemma \ref{lemma:func_b prime}.

(i)
Assume that $0<\kappa<\frac12$ and $a>G(\kappa)$.
In this case, we have 
\[
\begin{array}{*{5}{c|}c}
\theta&0&\cdots&\varphi_*&\cdots&\pi\\ \hline
b'(\theta)&&+&0&-&\\ \hline
b(\theta)&b(0)&\nearrow&&\searrow&b(\pi)>0
\end{array}
\]
If $b(0)\ge 0$, then we see that $b(\theta)>0$ for any $\theta\in I_0$,
which implies $r_\pm(\theta)<0$ for any $\theta\in I_D$.
If $b(0)<0$, then there exists a unique $0<\varphi<\varphi_*$ such that $b(\varphi)=0$,
and we have $D(0)<0$ by \eqref{ineq:prop2-2}.
Hence,
\[
\begin{array}{*{8}{c|}c}
\theta&0&\cdots&\varphi&\cdots&\varphi_*&\cdots&\pi\\ \hline
b'(\theta)&&+&+&+&0&-&\\ \hline
b(\theta)&-&\nearrow&0&\nearrow&+&\searrow&+\\ \hline\hline
D'(\theta)&&-&0&+&0&-&\\ \hline
D(\theta)&D(0)<0&\searrow&-&\nearrow&+&\searrow&1
\end{array}
\]
This table yields that $b$ is negative on $I_D$, whence $r_\pm(\theta)<0$ for any $\theta\in I_D$.

(ii) Assume that $0<\kappa<\frac12$ and $a\le G(\kappa)$.
Then, Lemma~\ref{lemma:signature b} tells us that $b'(\theta)<0$ for any $\theta\in I_0$.
Thus, 
$b(\theta)$ is monotonic decreasing on the interval $I_0$ with $b(\pi)>0$, 
and hence $b(\theta)>0$ for any $\theta\in I_0$.
This means that $D'(\theta)<0$ and $D(\theta)$ is monotonic decreasing on $I_0$.
Since $D(\pi)=1$, we see that $I_D=I_0$ and hence $r_\pm(\theta)<0$ for any $\theta\in I_D$.

(iii) Assume that $\frac12\le \kappa<1$.
In this case,
we have we have $b'(\theta)>0$ for any $\theta\in I_0$ so that $b(\theta)$ is monotonic increasing.
In fact,
if $\kappa>\frac12$,
then since $\frac{(1+\sqrt{2})^2}{6}<G(\kappa)<1$ for $\frac12<\kappa<1$
and since $a>1$,
we always have $a>G(\kappa)$ so that $b'(\theta)>0$ by Lemma~\ref{lemma:signature b},
and if $\kappa=\frac12$ then we have $b'(\theta)2(a-1)\sin\theta$ so that $b'(\theta)>0$ b y\eqref{eq:b prime k12}.
If $b(0)\ge 0$, then we have $b(\theta)\ge 0$ for any $\theta\in\theta$ and hence
we see that $r_\pm(\theta)<0$ for $\theta\in I_D$.
If $b(0)<0$ then there exists a unique $\varphi$ such that $b(\varphi)=0$ and we have
\[
\begin{array}{*{5}{c|}c}
\theta&0&\cdots&\varphi&\cdots&\pi\\ \hline
b'(\theta)&&+&+&+&\\ \hline
b(\theta)&-&\nearrow&0&\nearrow&+\\ \hline \hline
D'(\theta)&&-&0&+&\\ \hline
D(\theta)&-&\searrow&-&\nearrow&1
\end{array}
\]
This table indicates that for $\theta\in I_D$ we have $b(\theta)>0$, which implies $r_\pm(\theta)<0$.

We shall determine an asymptotic line with respect to $\Omega$ when $r_+(\theta)\to+\infty$ as $\theta\to\frac{\pi}{\kappa}$  or $\frac{\pi}{\kappa+1}$.
To calculate them in a one scheme, we set $\vartheta=\frac{\pi}{\kappa}$ or $\frac{\pi}{\kappa+1}$,
and denote its denominator by $\alpha$.
A line having gradient $\tan\vartheta$
can be written as $x\sin\vartheta-y\cos\vartheta=A$ with some constant $A$.
Since $x=\kappa(r(\theta)\cos\theta-1)$ and $y=\kappa r(\theta)\sin\theta$, we have
\[
\begin{array}{r@{\ }c@{\ }l}
x\sin\vartheta-y\cos\vartheta
&=&
\kappa\bigl\{\sin\vartheta(r(\theta)\cos\theta-1)-\cos\vartheta r(\theta)\sin\theta\bigr\}\\
&=&
\kappa\bigl\{r(\theta)(\cos\theta\sin\vartheta-\sin\theta\cos\vartheta)-\sin\vartheta\bigr\}\\
&=&
-\kappa\bigl\{r(\theta)\sin(\theta-\vartheta)+\sin\vartheta\bigr\}.
\end{array}
\]
When $\theta\to\vartheta$, we have $r_+(\theta)\to+\infty$ and $b(\theta)\to-\infty$, and
\[
r_+(\theta)=\frac{-b(\theta)+\sqrt{b(\theta)^2-4a(a-1)}}{2a}=-\frac{b(\theta)}{2a}\left(1+\sqrt{1-\frac{4a(a-1)}{b(\theta)^2}}\right).
\]
Since
\[
\lim_{\theta\to\vartheta-0}\frac{\sin(\theta-\vartheta)}{\sin(\alpha\vartheta)}
=
\lim_{\theta\to\vartheta-0}-\frac{\sin(\theta-\vartheta)}{\sin(\alpha(\theta-\vartheta))}
=
\lim_{\theta\to\vartheta-0}-\frac{\theta-\vartheta}{\alpha(\theta-\vartheta)}
=
-\frac1\alpha,
\]
we have
\[
\lim_{\theta\to\vartheta-0}
b(\theta)\sin(\theta-\vartheta)
=
\lim_{\theta\to\vartheta-0}
\Bigl(
(1-2a)\cos\theta\sin(\theta-\vartheta)+\sin\theta\cos(\alpha\theta)\frac{\sin(\theta-\vartheta)}{\sin(\alpha\theta)}\Bigr)
=
\frac{\sin\vartheta}{a\alpha},
\]
whence
\[
\lim_{\theta\to\vartheta-0}
r(\theta)\sin(\theta-\vartheta)
=
\lim_{\theta\to\vartheta-0}
-\frac{b(\theta)\sin(\theta-\vartheta)}{2a}\left(1+\sqrt{1-\frac{4a(a-1)}{b(\theta)^2}}\right)
=
-\frac{\sin\vartheta}{a\alpha}.
\]
Thus, we have
\[
A=\begin{cases}
-\kappa(-\frac{\sin\theta_0}{a\kappa}+\sin\theta_0)=\Bigl(\frac1a-\kappa\Bigr)\sin\theta_0
&
(\text{if }\alpha=\kappa),\\
-\kappa(-\frac{\sin\theta_1}{\kappa+1}+\sin\theta_1)
=
-\frac{\kappa^2}{\kappa+1}\sin\theta_1
&
(\text{if }\alpha=\kappa+1),
\end{cases}
\]
and therefore the proof is now complete.
\end{proof}

\section{The domain $\Omega$ for $\kappa=+\infty$}

In this section, we deal with the case $\kappa=\infty$.
Since $\kappa\theta(x,y)$ is regarded as $y$ in this case,
the equation \eqref{im} can be written as
\begin{equation}\label{eq:im k infty}
F(x,y)=x+\gamma x^2+\gamma y^2+y\cot y=0.
\end{equation}

Assume that $\gamma=0$, that is, the case of the classical Lambert function.
Then, the above equation reduces to $F(x,y)=x+y\cot y=0$. 
Since the point $z=x+yi=0$ satisfies $F(0,0)=1>0$,
we have $\Omega=\set{z=x+yi\in\C}{x>-y\cot y,\ |y|<\pi}$.

Assume that $\gamma\ne 0$. 
Then, the equation \eqref{eq:im k infty} can be calculated as
\begin{equation}\label{eq:xy}
x^2+\frac{x}{\gamma}+y^2+\frac{y\cot y}{\gamma}=0
\iff 
\Bigl(x+\frac{1}{2\gamma}\Bigr)^2=\frac{1}{4\gamma^2}-y^2-\frac{y\cot y}{\gamma}.
\end{equation}
Let us consider the function
\[
\begin{array}{r@{\ }c@{\ }l}
\ds h(y)&:=&\ds \frac{1}{4\gamma^2}-y^2-\frac{y\cot y}{\gamma}
=
\frac{1}{4\gamma^2}-\Bigl(y+\frac{\cot y}{2\gamma}\Bigr)^2+\frac{\cot^2y}{4\gamma^2}
=
\frac{1}{4\gamma^2\sin^2y}-\Bigl(y+\frac{\cot y}{2\gamma}\Bigr)^2\\
&=&
\ds
\frac{1-\bigl(2\gamma y\sin y+\cos y\bigr)^2}{4\gamma^2\sin^2y}.
\end{array}
\]
Note that
\[h(0)=\lim_{y\to0}h(y)=\frac{1}{4\gamma^2}-\frac{1}{\gamma}\lim_{y\to0}\frac{y}{\sin y}=
\frac{1-4\gamma}{4\gamma^2}\quad\text{and}\quad
\lim_{y\to\pi-0}|h(y)|=+\infty.
\]
In order that the equation~\eqref{eq:im k infty} has a real solution in $x$ and $y$, 
the function $h(y)$ needs to be non-negative,
and it is equivalent to the condition that the absolute value of the function $g(y):=\cos y+2\gamma y\sin y$ is less than or equal to $1$.
At first, 
we observe that
$g(0)=1$ and $g(\pi)=-1$, and its derivative is
\[
g'(y)=-\sin y+2\gamma(\sin y+y\cos y)=-(1-2\gamma)\sin y+2\gamma y\cos y
=
(2\gamma-1)\Bigl(\frac{2\gamma}{2\gamma-1}y+\tan y\Bigr)\cos y.
\]
Set $c_\gamma=\frac{2\gamma}{2\gamma-1}$.
Then, 
the signature of $g'$ can be determined by the signatures of $2\gamma-1$, $\cos y$ and $c_\gamma y+\tan y$.

Assume that $\gamma>\frac14$.
If $\gamma>\frac12$, then we have $c_\gamma>1$ and hence there exists a unique $y_*\in(\frac\pi2,\pi)$ such that $c_\gamma y+\tan y=0$ and we have
\[
\begin{array}{*{7}{c|}c}
y&0&\cdots&\frac\pi2&\cdots&y_*&\cdots&\pi\\ \hline
g'&&+&&+&0&-&\\ \hline
g&1&\nearrow&&\nearrow&&\searrow&-1
\end{array}
\]
If $\frac14<\gamma<\frac12$,
then we have $c_\gamma<-1$ and hence there exists a unique $y_*\in(0,\frac\pi2)$ such that $c_\gamma y_*+\tan y_*=0$ and we have
\[
\begin{array}{*{7}{c|}c}
y&0&\cdots&y_*&\cdots&\frac\pi2&\cdots&\pi\\ \hline
g'&&+&&-&&-&\\ \hline
g&1&\nearrow&&\searrow&&\searrow&-1
\end{array}\]
If $\gamma=\frac12$, then we have $g'(y)=y\cos y$ so that $g$ is monotonic decreasing on the interval $(\frac\pi2,\pi)$ with $g(\frac\pi2)>1$, and in this case we set $y_*=\frac\pi2$.
These observation shows that, if $\gamma>\frac14$, then there exists one and only one $y_0\in(y_*,\pi)$ such that
$g(y_0)=1$ and $g(y_0-\varepsilon)>1$ for $\varepsilon\in(0,y_0-y_*)$.
In this case,
$h(y)$ is non-negative in the interval $[y_0,\pi)$, and $h(h_0-\varepsilon)<0$ for $\varepsilon\in(0,y_0-y_*)$.
Let $x_i(y)$, $i=1,2$ be the real solutions of the equation~\eqref{eq:xy} with $x_1(y)\le x_2(y)$.
Then, since we have $x_1(y_0)=x_2(y_0)$ and 
\[
\lim_{y\to\pi-0}x_1(y)=-\infty,\quad
\lim_{y\to\pi-0}x_2(y)=+\infty,
\]
the curves $x_i(y)$, $y\in(y_0,\pi)$ form a connected curve, and hence we have
\[
\Omega=\set{z=x+yi\in\C}{|y|<\pi\text{ and if }|y|\ge y_0\text{ then }x<x_1(y)\text{ or }x>x_2(y)}.
\]

Assume that $0<\gamma\le \frac14$.
Then, we have $-1\le c_\gamma<1$ and hence there are no $y\in(0,\pi)$ such that $c_\gamma y+\tan y=0$.
Thus, we obtain $g'(y)<0$ for any $y\in(0,\pi)$ so that $g$ is monotonic decreasing from $g(0)=1$ to $g(\pi)=-1$.
This shows that $h(y)$ is non-negative in the interval $(0,\pi)$.
Let $x_i(y)$, $i=1,2$ be the solutions of the equation~\eqref{eq:xy} with $x_1(y)\le x_2(y)$.
Then, since we have
\[
\lim_{y\to\pi-0}x_1(y)=-\infty,\quad
\lim_{y\to\pi-0}x_2(y)=+\infty
\] 
and $x_1(0)<x_2(0)<0$,
we have
\[
\Omega=\set{z=x+yi\in\C}{|y|<\pi\text{ and }x>x_2(y)}.
\]

Assume that $\gamma<0$.
Then, there exists a unique $y_*\in(\frac\pi2,\pi)$ such that $c_\gamma y_*+\tan y_*=0$, 
and we have
\[
\begin{array}{*{7}{c|}c}
y&0&\cdots&\frac\pi2&\cdots&y_*&\cdots&\pi\\ \hline
g'&&-&&-&0&+&\\ \hline
g&1&\searrow&&\searrow&&\nearrow&-1
\end{array}
\]
This observation shows that
there exists one and only one $y_0\in(0,y_*)$ such that $g(y_0)=-1$ and $g(y_0+\varepsilon)<-1$ for $\varepsilon\in(0,y_*-y_0)$.
Thus, $h(y)$ is non-negative on $y\in[0,y_0]$, and $h(y_0+\varepsilon)<0$ for $\varepsilon\in(0,y_*-y_0)$.
Let $x_i(y)$, $i=1,2$ be the solutions of the equation~\eqref{eq:xy} with $x_1(y)\le x_2(y)$.
Then, since $x_1(0)<0<x_2(0)$ and $x_1(y_0)=x_2(y_0)$,
these two paths $(x_\pm(y),y)$ form a continuous curve connecting $x_+(0)$ and $x_-(0)$
and we have 
\[
\Omega=\set{z=x+yi\in\C}{|y|<y_0\text{ and }x_1(y)<x<x_2(y)}.
\]

We summarize these calculations as a proposition.

\begin{Proposition}
\label{prop:Omega k infty}
Assume that $\kappa=+\infty$ and let $x_i(y)$, $i=1,2$ be solutions of $\eqref{eq:xy}$ with $x_1(y)\le x_2(y)$ if they are real.
\begin{enumerate}[\quad\rm(1)]
    \item If $\gamma=0$, then one has
    $\Omega=\set{z=x+yi}{|y|<\pi\text{ and }x>-y\cot y}$.
    \item Suppose that $\gamma>\frac14$.
    In this case, there exists a unique $y_0\in(0,\pi)$ such that $x_1(y_0)=x_2(y_0)$ and if $y\ge y_0$ then $x_i(y)$, $i=1,2$ are real.
    Moreover, one has
    \[
    \Omega=\set{z=x+yi\in\C}{|y|<\pi\text{ and if }|y|\ge y_0\text{ then }x<x_1(y)\text{ or }x>x_2(y)}.
    \]
    \item If $0<\gamma\le \frac14$, then $x_i(y)$, $i=1,2$ are both real for any $y\in(0,\pi)$, 
    and one has
    \[
    \Omega=\set{z=x+yi\in\C}{|y|<\pi\text{ and }x>x_2(y)}.
    \]
    \item Suppose that $\gamma<0$.
    Then, there exists a unique $y_0\in(0,\pi)$ such that $x_1(y_0)=x_2(y_0)$ and 
    if $y\le y_0$ then $x_i(y)$, $i=1,2$ are real for $\eqref{eq:xy}$.
    Moreover, one has
    \[
    \Omega=\set{z=x+yi\in\C}{|y|<y_0\text{ and }x_1(y)<x<x_2(y)}.
    \]
    In particular, $\Omega$ is bounded.
\end{enumerate}
\end{Proposition}

\section{Proof of Theorem \ref{theorem:lambert-tsallis}}

In this section,
we shall show Theorem~\ref{theorem:lambert-tsallis} for $\kappa>0$ or $\kappa=\infty$.
To do so,
we consider the domain $D:=\Omega\cap\C^+$,
and show that $f_{\kappa,\gamma}$ maps $D$ to $\C^+$ bijectively.

The key tool is the argument principle
(see \cite[Theorem 18, p.152]{Ahlfors}, for example).

\begin{Theorem}[{\bf The argument principle.}] 
If $f(z)$ is meromorphic in a domain $\Omega$ with the zeros $a_j$ and the poles
$b_k$, then
\[
\frac{1}{2\pi i}\int_\gamma \frac{f'(z)}{f(z)}\,dz=\sum_jn(\gamma,a_j)-\sum_kn(\gamma,b_k)
\]
for every cycle $\gamma$ which is homologous to zero in $\Omega$ and does not pass through
any of the zeros or poles.
Here,
$n(\gamma,a)$ is the winding number of $\gamma$ with respect to $a$.
\end{Theorem}

We also use the following elementary property of holomorphic functions.

\begin{Lemma}
\label{lemm:implicit}
Let $f(z)=u(x,y)+iv(x,y)$ be a holomorphic function.
The implicit function $v(x,y)=0$ has an intersection point at $z=x+yi$ only if $f'(z)=0$.
\end{Lemma}
\begin{proof}
Let $p(t)=(x(t),y(t))$ be a continuous path in $\C\cong\R^2$ satisfying $v\bigl(p(t)\bigr)=0$ for all $t\in[0,1]$.
We {assume that} $(x'(t),y'(t))\ne(0,0)$.
Set
\[
g(t):=u(p(t))=u(x(t),y(t)),\quad
h(t):=v(p(t))=v(x(t),y(t)).
\]
Obviously, we have $h'(t)\equiv0$ for any $t$, and
\[
h'(t)=v_xx'(t)+v_yy'(t)=(v_x,v_y)\cdot(x'(t),y'(t)).
\]
Assume that $g'(t_0)=0$ for some point $t_0\in[0,1]$.
Then
\[
\begin{array}{r@{\ }c@{\ }l}
g'(t)
&=&
\ds
u_xx'(t)+u_yy'(t)=(u_x,u_y)\cdot(x'(t),y'(t))\\
&=&
\ds
(v_x,v_y)\pmat{0&-1\\1&0}\cdot(x'(t),y'(t))
=
(v_x,v_y)\cdot (-y'(t),x'(t)),
\end{array}
\]
{the condition $g'(t_0)=0$ implies that }
the vector $(v_x,v_y)$ is orthogonal both to $(x'({t_0}),y'({t_0}))$ and $(-y'({t_0}),x'({t_0}))$, which are non-zero vectors and mutually orthogonal.
Such vector is only zero vector in $\R^2$, that is, $(v_x,v_y)=(0,0)$,
and hence $(u_x,u_y)=(0,0)$ by {Cauchy-Riemann equations}.
Thus, if $g'(t_0)=0$ then $p(t_0)$ needs to {satisfy} {$f'(p(t_0))=0$}.
\end{proof}

\subsection{The case $\gamma<0$}

Assume that $\gamma<0$.
By Propositions \ref{prop:Omega k=1}, \ref{prop:Omega} and \ref{prop:Omega k infty},
it is equivalent to the condition that $\Omega$ is bounded.
We first consider the set $\mathcal{S}$
(see \eqref{def:mathcal S} for definition).
Let $\alpha_i$, $i=1,2$ be solutions of \eqref{eq:quad}.
Since $\gamma<0$, these are distinct real numbers.
Set $\alpha_1< \alpha_2$.
Then, we have $\Omega\cap\R=(\alpha_1,\alpha_2)$.

\begin{Lemma}
One has $f(\alpha_2)<f(\alpha_1)<0$ and 
$\mathcal{S}=(f_{\kappa,\gamma}(\alpha_2),f_{\kappa,\gamma}(\alpha_1))$.
\end{Lemma}

\begin{proof}
Assume that $\kappa<+\infty$.
Since $\gamma<0$, we have by \eqref{eq:derivative f}
\[
\begin{array}{*{10}{c|}c}
x&-\kappa&\cdots&\alpha_1&\cdots&0&\cdots&-\frac1\gamma&\cdots&\alpha_2&\cdots\\ \hline
f'_{\kappa,\gamma}&&-&0&+&+&+&\times&+&0&-\\ \hline
f_{\kappa,\gamma}&&\searrow&f_{\kappa,\gamma}(\alpha_1)&\nearrow&0&\nearrow^{+\infty}&\times&{}_{-\infty}\nearrow&f_{\kappa,\gamma}(\alpha_2)&\searrow
\end{array}
\]
The inequality $f(\alpha_1)<0$ is obvious by the above table.
We shall show $f(\alpha_1)>f(\alpha_2)$.
By the fact that $\alpha_1\alpha_2=\frac{1}{\gamma}$,
we have
\[
\frac{f(\alpha_2)}{f(\alpha_1)}
=
\frac{\alpha_2(1+\gamma \alpha_1)}{(1+\gamma\alpha_2)\alpha_1}\cdot\biggl(\frac{1+\alpha_2/\kappa}{1+\alpha_1/\kappa}\biggr)^\kappa
=
\frac{\alpha_2+1}{\alpha_1+1}\cdot\biggl(\frac{1+\alpha_2/\kappa}{1+\alpha_1/\kappa}\biggr)^\kappa.
\]
Since $1+\gamma\alpha_2<0$ and $\alpha_1<0$,
we have $\alpha_1+1=(1+\gamma\alpha_2)\alpha_1>0$.
Moreover, the facts that $1+\alpha_1/\kappa {>0}$ and  $\alpha_2>\alpha_1$ yield that
\[
\frac{\alpha_2+1}{\alpha_1+1}>1\quad\text{and}\quad
\frac{1+\alpha_2/\kappa}{1+\alpha_1/\kappa}>1,
\]
whence we obtain $\ds\frac{f(\alpha_2)}{f(\alpha_1)}>1$.
Since $f(\alpha_2)<0$ because $\alpha_2>-\frac{1}{\gamma}$ and $\gamma<0$,
we conclude that $0>f(\alpha_1)>f(\alpha_2)$.

Assume that $\kappa=+\infty$.
Since $\gamma<0$ and $\gamma(-\frac1\gamma)^2+(-\frac1\gamma)+1=1>0$, we have the following 
variation table of $f(x)$:
\[
\begin{array}{c*{11}{|c}}
x&-\infty&\cdots&\alpha_1&\cdots&0&\cdots&-\frac1\gamma&\cdots&\alpha_2&\cdots&+\infty\\ \hline
f'&\multicolumn{2}{|c|}{-}&0&\multicolumn{3}{|c|}{+}&\times&+&0&\multicolumn{2}{|c}{-}\\ \hline
f&0&\searrow&f(\alpha_1)&\nearrow&0&\nearrow&\times&\nearrow&f(\alpha_2)&\searrow&-\infty
\end{array}
\]
Since $\gamma\alpha_i+1=-\frac{1}{\alpha_i}$, we see that $f(\alpha_i)=-\alpha_i^2e^{\alpha_i}<0$.
By $\alpha_1\alpha_2=\frac{1}{\gamma}$, we have
\[
\frac{f(\alpha_2)}{f(\alpha_1)}
=
\frac{\alpha_2(1+\gamma\alpha_1)}{\alpha_1(1+\gamma\alpha_2)}e^{\alpha_2-\alpha_1}
=
\frac{\alpha_2+1}{\alpha_1+1}e^{\alpha_2-\alpha_1}>1,
\]
whence $f(\alpha_2)<f(\alpha_1)<0$.
Thus, the proof is now completed.
\end{proof}

Now we show that $f_{\kappa,\gamma}\colon D\to\C^+$ is bijective.
Since the proof is completely analogous, we only prove the case $\kappa<+\infty$.

We take a path $C=C(t)$ $(t\in {[}0,1 {]})$ in such a way that
by starting from $z=-\frac{1}{\gamma}$,
it goes to $z=\alpha_2$ along the real axis,
next goes to $z=\alpha_1$ along the  curve
{$r_{+-}$ defined by \eqref{im} and} connecting $\alpha_2$ and $\alpha_1$ 
in the upper half plane,
and then it goes to $z=-\frac{1}{\gamma}$ along the real axis
(see Figure~{\ref{fig:path(i)}}).
Here, we can assume that $C'(t)\ne0$ whenever $C(t)\ne\alpha_i$, $i=1,2$.
Actually,
the curve $v(x,y)=0$ has a tangent line unless $f'$ vanishes.
If we take an arc-length parameter $t$,
then $C'(t)$ represents the direction of the tangent line at $(x,y)=C(t)$.
We note that $C(t)$ describes the boundary of $D$.

We show that $f_{\kappa,\gamma}$ maps the boundary of $D$ to $\R$ bijectively.
We take $t_i$, $i=1,2$ as $C(t_i)=\alpha_i$.
Note that the sub-curve $C(t)$, $t\in(t_2,t_1)$ describes the curve $r_{+-}(t)$,
and $f_{\kappa,\gamma}$ does not have a pole or singular point on $C(t)$, $t\in(t_2,t_1)$.
Set $f(z)=u(x,y)+iv(x,y)$.
By Lemma~\ref{lemm:implicit},
the implicit function $v(x,y)=0$ may have an intersection point only if $f'(x+iy)=0$, i.e. at $x+iy=\alpha_i$ $(i=1,2)$ or at  $x+iy=-\kappa$ if $\kappa>1$.
Then, the function $g(t)=u(C(t))$, $t\in[t_2,t_1]$ attains  maximum and minimum  in the interval 
because it is a continuous function on a compact set.
Moreover, $g'$ never vanishes in $(t_2,t_1)$ by the above argument
and by the fact that $f'(C(t))\ne 0$ for $t\in(t_2,t_1)$.
Therefore, $g$ is monotone and hence it takes maximal and minimal values at the endpoints $t=t_2,t_1$.
Now we have $f(\alpha_1)>f(\alpha_2)$ by the last claim so that 
the image of $g$ is $[f(\alpha_2),f(\alpha_1)]$, and the function $g$ is bijective.

We shall show that for any $w_0\in\C^+$ there exists one and only one $z_0\in D$ such that
$f(z_0)=w_0$.
Let us take an $R>0$ such that $|w_0|<R$.
For $\delta>0$,
let ${C'=}C_{\delta}$ be a path obtained from $C$ in such a way that
the pole $z=-1/\gamma$
is avoided by a semi-circle $-\tfrac{1}{\gamma}+\delta e^{i\theta}$, $\theta\in(0,\pi)$ of radius $\delta$
(see Figure \ref{fig:mpath(i)}).
{Denote by $D'$ the domain surrounded by the curve $C'$.}

Then, we can choose $\delta>0$ such that
\[\Bigl|f\bigl(-\tfrac{1}{\gamma}+\delta e^{i\theta}\bigr)\Bigr|>R
\quad
(\text{for all }\theta\in(0,\pi)).
\]
In fact, if $z=-\tfrac{1}{\gamma}+\delta e^{i\theta}$, then we have
\[
\bigl|1+\gamma z\bigr|={|\gamma|}\delta,\quad
\bigl|z\bigr|=\bigl|-\tfrac{1}{\gamma}+\delta e^{i\theta}\bigr|
> \frac{1}{2|\gamma|}\quad(\text{if } \delta<\tfrac{1}{2|\gamma|}),
\]
and
\[
\bigl|1+\tfrac{z}{\kappa}\bigr|
=
\bigl|1-\tfrac{1}{\kappa\gamma}+\tfrac{\delta}{\kappa}e^{i\theta}\bigr|
>
\frac{\kappa\gamma-1}{2\kappa\gamma}\quad
(\text{if }\delta<\tfrac{\kappa}{2}\bigl|1-\tfrac{1}{\kappa\gamma}\bigr|),
\]
so that
\[
\Bigl|f\bigl(-\tfrac{1}{\gamma}+\delta e^{i\theta}\bigr)\Bigr|
>
\frac{1}{2|\gamma|^{ 2}}\Bigl(\frac{\kappa\gamma-1}{2\kappa\gamma}\Bigr)^\kappa\cdot
\frac{1}{\delta}.
\]
Thus it is enough to take \[\delta=\min\Bigl(\tfrac{1}{2|\gamma|^2R}\bigl(\tfrac{\kappa\gamma-1}{2\kappa\gamma}\bigr)^\kappa,\,\tfrac{1}{2|\gamma|},\,\tfrac{\kappa}{2}\bigl|1-\tfrac{1}{\kappa\gamma}\bigr|\Bigr).\]
Since $f$ is non-singular on the semi-circle $-\tfrac{1}{\gamma}+\delta e^{i\theta}$, $\theta\in[0,\pi]$,
the curve $\theta\mapsto f(-\frac{1}{\gamma}+\delta e^{i\theta})$ does not have a singular {angular} point,
so that
it is {homotopic}  to 
a large semicircle (with radius  larger than $R$) in the upper half-plane
(see Figure \ref{fig:fpath(i)}).

Note that
\[
\mathrm{Im}\,f(x+yi)=
\frac{
\Bigl((1+x/\kappa)^2+(y/\kappa)^2\Bigr)^{\tfrac{\kappa}{2}}
}{(1+\gamma x)^2+\gamma^2 y^2}
\bigl\{
    (x+\gamma x^2+\gamma y^2)\sin(\kappa\theta(x,y)) + y\cos(\kappa\theta(x,y))
\bigr\}.
\]
By changing variables as in \eqref{eq:v change}, we have
\[
\begin{array}{r@{\ }c@{\ }l}
\mathrm{Im}\,f(re^{i\theta})
&=&
\text{positive factor}\times \sin(\kappa\theta)\cdot(a r^2+b(\theta)r+a-1)\\
&=&
\text{positive factor}\times \sin(\kappa\theta)\cdot a(r-r_-(\theta))(r-r_+(\theta)).
\end{array}
\]
Note that the inside of the path $C$ can be written as $\set{re^{i\theta}}{\theta\in(0,\theta_*),\ r\in(r_-(\theta),r_+(\theta))}$ in $(r,\theta)$ coordinates.
Since $a<0$ and $\sin(\kappa\theta)>0$ when $\theta\in(0,\theta_*)$,
we see that $\mathrm{Im}\,f(z)>0$ if $z$ is  inside of the path $C$.
In particular, the inside set of the curve $f(C')$ is a bounded domain in $\C^+$ including $w_0$.

Since the winding number of the path $f(C')$ with respect to $w=w_0$ is exactly one,
we see that
\[
\frac{1}{2\pi i}\int_{C'}\frac{f'(z)}{f(z)-w_0}dz
=
\frac{1}{2\pi i}\int_{f(C')}\frac{dw}{w-w_0}=1.
\]

{By definition of $f$}, we see that $f(z)-w_0$ does not have a pole in $D'$.
Therefore, by the argument principle, 
the function $f(z)-w_0$ has  only one zero point, say $z_0{\in D'\subset D}$.
Thus, we obtain $f(z_0)=w_0$, and such $z_0{\in D}$ is unique.
{We conclude that} the map $f$ is a bijection from {$D$} to the upper half-plane ${\C^+}$.

\subsection{The case $\gamma\ge 0$}
Assume that $\gamma\ge0$.
By Propositions~\ref{prop:Omega}, \ref{prop:Omega k=1} and \ref{prop:Omega k infty},
it is equivalent to the condition that $\Omega$ is unbounded.
We first consider the behavior of $|f_{\kappa,\gamma}(z)|$ as $\partial\Omega\ni z\to\infty$.
We have
\[
\bigl|f_{\kappa,\gamma}(z)\bigr|
=
\frac{1}{\bigl|\gamma+\frac1z\bigr|}\cdot\Bigl|1+\frac{z}{\kappa}\Bigr|^\kappa.
\]
We consider the change variables $1+\frac{z}{\kappa}=re^{i\theta}$ for $z\in\Omega$.
Then, we have
\[
\Bigl|\gamma+\frac1z\Bigr|
\le |\gamma|+\frac{1}{|\kappa|}\cdot\frac{1}{|re^{i\theta}-1|}
\le |\gamma|+\frac{1}{|\kappa|\cdot|r-1|},
\]
so that
\[
\bigl|f_{\kappa,\gamma}(z)\bigr|
\ge \frac{r^\kappa}{|\gamma|+\frac{1}{|\kappa|\cdot|r-1|}}\quad
\longrightarrow
\quad
+\infty\quad(\text{as }r\to+\infty).
\]
Propositions~\ref{prop:Omega}, \ref{prop:Omega k=1} and \ref{prop:Omega k infty} show that,
if $\kappa>1$ or $\gamma=0$,
then $\partial\Omega\cap\C^+$ can be described as a connected curve $C(t)$, $t\in[0,1)$ with $C(0)=\alpha_2$ and $\lim_{t\to1-0}C(t)=\infty$.
In this case,
if $z\in\partial\Omega\cap\C^+$,
then we have $F(x,y)=0$ and
\begin{equation}
\label{eq:fkg on boundary}
    f_{\kappa,\gamma}(z)=\frac{\bigl|1+\frac{z}{\kappa}\bigr|^\kappa}{\bigl|1+\gamma z\bigr|^2}\cdot \frac{-y}{\sin(\kappa\theta(x,y))}<0.
\end{equation}
This means that,
if $z\in\C^+$ goes to $\infty$ along the path $\partial\Omega\cap\C^+$,
then $f_{\kappa,\gamma}$ must tend to $-\infty$.

We next consider $\mathcal{S}$ for cases $\kappa\ge 1$ or $\gamma=0$.

\begin{Lemma}
Let $\alpha_i$, $i=1,2$ be the solutions of $\eqref{eq:r}$.
\begin{enumerate}[\rm(1)]
    \item Assume that $\kappa>1$ or $\kappa=\infty$, and also $D(0)\ge 0$.
    Then, $\alpha_i$ are both real and $\mathcal{S}=(-\infty,f_{\kappa,\gamma}(\alpha_2))$ with $f_{\kappa,\gamma}(\alpha_2)<0$.
    \item Assume that $\gamma=0$.
    Then, one has $\alpha_i=-\frac{\kappa}{\kappa+1}$ if $\kappa<+\infty$ and $\alpha_i=-1$ if $\kappa=\infty$ $(i=1,2)$. 
    Moreover, one has $\mathcal{S}=(-\infty,f_{\kappa.\gamma}(\alpha_2))$.
    \item Assume that $\kappa=1$ and $0<\gamma<1$.
    Then, $\alpha_i$ are both real, and one has 
    $\mathcal{S}=(f_{\kappa,\gamma}(\alpha_1),f_{\kappa,\gamma}(\alpha_2))$ 
    with $f_{\kappa,\gamma}(\alpha_2)<0$.
    \item Assume that $\kappa\ge 1$ or $\kappa=\infty$, and also $D(0)<0$.
    Then, $\alpha_i$ are both non-real, and one has $\mathcal{S}=\emptyset$.
    On the other hand,
    one has $f_{\kappa,\gamma}(\partial\Omega\cap\C^+)=(-\infty,0)$.
\end{enumerate}
\end{Lemma}

\begin{proof}
(1) If $\kappa>1$ and $D(0)\ge 0$,
then 
Proposition \ref{prop:Omega} (5-a) tells us that $\Omega\cup\R=(\alpha_2,+\infty)$.
On the other hand,
if $\kappa=\infty$ and $D(0)\ge 0$, that is, $\gamma\le \frac14$,
then we also have $\Omega\cup\R=(\alpha_2,+\infty)$
by Proposition \ref{prop:Omega k infty} (1) and (3).
Since it is easily verified that $f_{\kappa,\gamma}$ is monotonic increasing on this interval for both cases, 
we have $f_{\kappa,\gamma}(\alpha_2)<0$ and $\ds\lim_{x\to+\infty}f_{\kappa,\gamma}(x)=+\infty$.
Thus,
by definition \eqref{def:mathcal S},
we obtain $\mathcal{S}=(-\infty,f_{\kappa,\gamma}(\alpha_2))$ for both cases.

\noindent
(2) Assume that $\gamma=0$.
Propositions \ref{prop:Omega} (2) and \ref{prop:Omega k infty} (1) show that
we have $\Omega\cap\R=(\alpha,+\infty)$ where $\alpha=\alpha_1=\alpha_2=-\frac{\kappa}{\kappa+1}$ if $\kappa<+\infty$, and if $\kappa=\infty$ then $\alpha=-1$.
Since it is easily verified that $f_{\kappa,\gamma}$ is monotonic increasing on this interval, 
we have $f_{\kappa,\gamma}(\alpha_2)<0$ and $\ds\lim_{x\to+\infty}f_{\kappa,\gamma}(x)=+\infty$.
Thus,
by definition \eqref{def:mathcal S},
we obtain $\mathcal{S}=(-\infty,f_{\kappa,\gamma}(\alpha))$.

\noindent
(3) Assume that $\kappa=1$ and $0<\gamma<1$.
By Proposition \ref{prop:Omega k=1},
we have $\Omega\cap\R=(-\infty,\alpha_1)\cup(\alpha_2,+\infty)$.
An elementary calculation yields that
the image of this set by $f_{\kappa,\gamma}$ is $(-\infty,f_{\kappa,\gamma})\cup (f_{\kappa,\gamma}(\alpha_2),+\infty)$ with $f_{\kappa,\gamma}(\alpha_1)<f_{\kappa,\gamma}(\alpha_2)<0$,
and therefore we obtain $\mathcal{S}=(f_{\kappa,\gamma}(\alpha_1),f_{\kappa,\gamma}(\alpha_2))$.

\noindent
(4) Assume that $\kappa\ge 1$ or $\kappa=\infty$, and also assume that $D(0)<0$.
Then,
Propositions \ref{prop:Omega} (5-b) and \ref{prop:Omega k infty} (2) tell us that $\Omega\cap\R=(-\kappa,+\infty)$, where if $|kappa=\infty$ then we regard $-\kappa$ as $-\infty$.
Since we have $-\kappa<-\frac1\gamma<0$,
we have the following increasing/decreasing table of $f_{\kappa,\gamma}$.
\[
\begin{array}{*{6}{c|}c}
x&-\kappa&\cdots&-\frac1\gamma&\cdots&0&\cdots\\ \hline
f'_{\kappa,\gamma}&&+&\times&+&+&+\\ \hline
f_{\kappa,\gamma}&0&\nearrow^{+\infty}&\times&{}_{-\infty}\nearrow&0&\nearrow^{+\infty}
\end{array}
\] 
Thus, we obtain $\mathcal{S}=\emptyset$.
Let $C(t)$, $t\in[0,1)$ be a path describing $\partial\Omega\cap\C^+$ with $C(0)=\alpha_2$.
Assume that $C(t_1)=\alpha_1$.
Then, by \eqref{eq:fkg on boundary} and by the discussion below of it,
we see that $f_{\kappa,\gamma}(\alpha_1)<0$ and $\lim_{t\to1}f_{\kappa,\gamma}(C(t))=-\infty$.
Thus, by Lemma~\ref{lemm:implicit}, 
we conclude that $f_{\kappa,\gamma}(\partial\Omega\cap\C^+)=(-\infty,0)$.
\end{proof}

We note that if $(\kappa,\gamma)=(1,1)$, then we have $f_{\kappa,\gamma}(z)=z$, and hence we omit this case.
Recall that $D=\Omega\cap\C^+$.
We shall show that $f_{\kappa,\gamma}$ maps $D$ to $\C^+$ bijectively.
We divide cases according to the above lemma.

\subsubsection{The case $(1)$ and $(2)$}

Assume that $\kappa>1$ or $\gamma=0$.
We also assume that $D(0)\ge 0$.
Let us take a path $C=C(t)$, $t\in(0,1)$ in such a way that
by starting from $z=\infty$,
it goes to $z=\alpha_2$ along the curve {$r_+$ defined by} \eqref{im} in the upper half plane,
and then goes to $z=\infty$ along the real axis
(see Figure {\ref{fig:path(ii)}}).
Here, we can assume that $C'(t)\ne0$ whenever $C(t)\ne\alpha_i$, $i=1,2$.
Actually,
the curve $v(x,y)=0$ has a tangent line unless $f'$ vanishes.
If we take an arc-length parameter $t$,
then $C'(t)$ represents the direction of the tangent line at $(x,y)=C(t)$.
We note that $C(t)$ describes the boundary of $D$.
Lemma shows that 
$g(t):=f_{\kappa,\gamma}(C(t))$ is a monotonic increasing function on $(0,1)$ such that 
$g(t)\to-\infty$ if $t\to 0$ and $g(t)\to+\infty $ if $t\to1$.

We shall show that for any $w_0\in\C^+$ there exists one and only one $z_0\in D$ such that
$f(z_0)=w_0$.
Note that 
we have $\mathrm{Im}\,f_{\kappa,\gamma}(z)>0$ for any $z\in D$
by definition.
Let us take an $R>0$ such that $|w_0|<R$.
For $L>0$,
let $\Gamma_L$ be the the circle
$-\kappa+Le^{i\theta}$ of origin $z=-\kappa$ with radius $L$.
Let $L-\kappa$ and $z_L$ be two distinct intersection points of $C$ and $\Gamma_L$.
Let $C':=C_{L}$ be a closed path obtained from $C$ by connecting $L-\kappa$ and $z_L$ via the arc {$A$} of $\Gamma_L$ included in the upper half plane,
\out{and by avoiding the zero $z=0$ by a semi-circle $\delta e^{i\theta}$, $\theta\in(0,\pi)$ of radius $\delta$,}
see Figure \ref{fig:mpath(ii)}.

Since $f$ is non-singular on the arc {$A$},
the curve $f({A})$ does not have a singular point so that
it is 
homotopic to a large {semi}-circle (whose radius is larger than $R$) in the upper half plane
(see Figure~\ref{fig:fpath(ii)}).
In particular, the inside set
${f(D')}$ of the curve $f(C')$ is a bounded domain including $w_0\in\C^+$.
Since the winding number of the path $f(C')$ about $w=w_0$ is exactly one,
we see that
\[
\frac{1}{2\pi i}\int_{C'}\frac{f'(z)}{f(z)-w_0}dz
=
\frac{1}{2\pi i}\int_{f(C')}\frac{dw}{w-w_0}=1.
\]
We know by definition of $f$ that  $f$ does not have a pole on $D'$.
Therefore, by the argument principle, 
the function $f(z)-w_0$ has the only one zero point, say $z_0\in D'$.
Then, we obtain $f(z_0)=w_0$, and such $z_0$ is unique.
We conclude that the map $f$ is bijection from the interior set $D$ of $C$ to the upper half plane.

\subsubsection{The case $(3)$}

Assume that $\kappa=1$ and $0<\gamma<1$.
Then, $\Omega$ is given in Proposition~\ref{prop:Omega k=1}.
By a simple calculation,
we see that $\alpha_i$, $i=1,2$ are both real such that $\alpha_1<\alpha_2<0$,
and we have $\Omega\cap\R=(-\infty,\alpha_1)\cup(\alpha_2,+\infty)$.
Moreover,
\[
f_{\kappa,\gamma}([\alpha_1,+\infty))=[f_{\kappa,\gamma}(\alpha_1),+\infty),\quad
f_{\kappa,\gamma}((-\infty,\alpha_2])=(-\infty,f_{\kappa,\gamma}(\alpha_2)],
\]
and we have $\mathcal{S}=(f_{\kappa,\gamma}(\alpha_1),f_{\kappa,\gamma}(\alpha_2))$.
Let $C=C(t)$, $t\in[0,1]$ be a path in $\C^+$ connecting $z=\alpha_1$ to $z=\alpha_2$ along $\partial\Omega\cap\C^+$.
Then, since $f_{\kappa,\gamma}'(C'(t))\ne 0$ for any $t\in(0,1)$ and since $f_{\kappa,\gamma}(C(t))\in\R$,
Lemma ~\ref{lemm:implicit} tells us that $f_{\kappa,\gamma}(C(t))$ is monotonic increasing on the interval $(0,1)$,
and hence $f_{\kappa,\gamma}(C([0,1]))=\overline{\mathcal{S}}$.
This shows that $f_{\kappa,\gamma}$ maps $\partial D$ to $\R$ bijectively.

We can show the bijectivity of $f_{\kappa,\gamma}$ on $D$ to $\C^+$ similarly to the case $\kappa>1$ and $D(0)\ge 0$
by taking $\Gamma_L$ to be the semi-circle $-1+Le^{i\theta}$ of origin $z=-1$ with radius $L$,
contained in the upper half plane.
Thus, we omit the detail.



\subsubsection{The case $(4)$}

Assume that $\kappa\ge 1$ and $D(0)<0$.
Then,
the solutions $\alpha_i$, $i=1,2$ of \eqref{eq:quad} are both non-real complex numbers by $D(0)<0$.
Let $\mathrm{Im}\,\alpha_1>0$.
In this case, we have $-\kappa<-\frac1\gamma<0$
and
\[
\begin{array}{*{6}{c|}c}
x&-\kappa&\cdots&-\frac1\gamma&\cdots&0&\cdots\\ \hline
f'_{\kappa,\gamma}&&+&\times&+&+&+\\ \hline
f_{\kappa,\gamma}&0&\nearrow^{+\infty}&\times&{}_{-\infty}\nearrow&0&\nearrow^{+\infty}
\end{array}
\] 
Note that if $\kappa=\infty$, then we regard $-\kappa$ as $-\infty$.
By the above table,
we have $\mathcal{S}=\emptyset$, and set $\mathcal{S}'=(-\infty,0)$.
Since we have for $z=x+yi\in\partial\Omega$
\[
f_{\kappa,\gamma}(z)=C((x+\gamma x^2+\gamma y^2)\cos(\kappa \theta(x,y))-y\sin(\kappa\theta(x,y)))
=-C\frac{y}{\sin(\kappa\theta(x,y))},
\]
we see that $f_{\kappa,\gamma}(\alpha_1)<0$.
This shows that the image of the function $f_{\kappa,\gamma}(z(\theta))$, $\theta\in(0,\theta_0)$ 
is $(-\infty,0)$.
Let us adopt a similar argument of the proof of bijectivity of the map $f_{\kappa,\gamma}$ as in unbounded cases.
since the winding number of the path of the boundary $\partial D$ is two,
we see that the map $f_{\kappa,\gamma}$ maps $D$ to $\C^+$ in two-to-one,
and hence  $f_{\kappa,\gamma}$ does not map $D$ to $\C^+$ bijectively.


\subsubsection{}

Assume that $0<\kappa<1$ and $\gamma>0$.
In this case, Proposition~\ref{prop:Omega} tells us that 
if $0<a=\kappa\gamma<1$ then $\Omega\cap\R=(\alpha_2,+\infty)$,
and if $a>1$ then $\Omega\cap\R=\set{x\in\R}{x>-\kappa\text{ and }x\ne-\frac1\gamma}$.
If $f_{\kappa,\gamma}$ maps $D=\Omega\cap\C^+$ to $\C^+$ bijectively,
then it needs map $\partial D=\mathcal{I}$ to $\R$.
However, if $x\in\R$ satisfies $x<\min(\alpha_1,-\kappa)$, 
then $\exp_\kappa(z)$ to $\bigl|1+\frac{x}{\kappa}\bigr|^\kappa\,e^{i\kappa\pi}$ 
as $z\to x$ via the arc $re^{i\theta}$ where $r=\bigl|1+\frac{x}{\kappa}\bigr|$.
Since $0<\kappa<1$, then we see that $\lim_{z\to x}f_{\kappa,\gamma}(z)$ is not real and hence
$f_{\kappa,\gamma}$ cannot map $D$ to $\C^+$ bijectively.

Now we have completed the proof of Theorem~\ref{theorem:lambert-tsallis} for the case $\kappa>0$ or $\kappa=\infty$.
\qed







\section{The case of $\kappa<0$}

We shall complete the proof of Theorem~\ref{theorem:lambert-tsallis} by proving it 
for the case $\kappa<0$.
To do so, let us recall the homographic (linear fractional) action of $SL(2,\R)$ on $\C$. 
For $\pmat{a&b\\c&d}\in SL(2,\R)$ and $z\in\C^+$, we set
\[
\pmat{a&b\\c&d}\cdot z:=\frac{az+b}{cz+d}.
\]
For each $g\in SL(2,\R)$,
the corresponding homographic action map $\C^+$ to $\C^+$ bijectively.
Let $\kappa=-\kappa'$ with positive $\kappa'>0$. Consider the transformation
\[
    1+\frac{z'}{\kappa'}=\Bigl(1+\frac{z}{\kappa}\Bigr)^{-1}.
\]
Then, it can be written as
\[
    z'=\pmat{1&0\\1/\kappa&1}\cdot z=\frac{z}{1+z/\kappa}
    \quad\Longleftrightarrow\quad
    z=\pmat{1&0\\-1/\kappa&1}\cdot z'=\frac{z'}{1-z'/\kappa}.
\]
Note that since $\pmat{1&0\\1/\kappa&1}\in SL(2,\R)$, it maps $\C^+$ to $\C^+$ bijectively.
Then, since
\[
    \begin{array}{r@{\ }c@{\ }l}
    \ds
    \frac{z}{1+\gamma z}=\pmat{1&0\\ \gamma&1}\cdot z
    &=&
    \ds
    \pmat{1&0\\ \gamma&1}\pmat{1&0\\-1/\kappa&1}\cdot z'
    =
    \pmat{1&0\\ \gamma-1/\kappa&1}\cdot z'\\
    &=&
    \ds
    \frac{z'}{1+(\gamma-1/\kappa)z'}
    =
    \frac{z'}{1+(\gamma+1/\kappa')z'}
    \end{array}
\]
and
\[
    \Bigl(1+\frac{z}{\kappa}\Bigr)^\kappa
    =
    \biggl(\Bigl(1+\frac{z}{\kappa}\Bigr)^{-1}\biggr)^{-\kappa}
    =
    \Bigl(1+\frac{z'}{\kappa'}\Bigr)^{\kappa'}
\]
(recall that  we are taking the main branch so that $\log z=-\log(z^{-1})$),
we obtain
\[
    f_{\gamma,\kappa}(z)
    =
    \frac{z}{1+\gamma z}
    \Bigl(1+\frac{z}{\kappa}\Bigr)^\kappa
    =
    \frac{z'}{1+(\gamma+1/\kappa')z'}
    \Bigl(1+\frac{z'}{\kappa'}\Bigr)^{\kappa'}
    =
    f_{\gamma+1/\kappa',\,\kappa'}(z').
\]
Set $\gamma'=\gamma+1/\kappa'$.
Since homographic actions map $\C^+$ to $\C^+$ bijectively,
there exists a domain $\Omega$ such that $f_{\kappa,\gamma}$ maps $D=\Omega\cap\C^+$ to $\C^+$ bijectively
if and only if it holds for $f_{\gamma',\kappa'}$.
Thus,
\[
\gamma'\le 0
\iff
\gamma\le \frac1\kappa,
\]
and $\kappa'>1$ and $\gamma'>0$ with $\gamma'\le \frac14(1+\frac{1}{\kappa'})^2$ is equivalent to
\[
\gamma>\frac1\kappa\quad\text{and}\quad\kappa<-1\quad\text{with}\quad
\gamma-\frac1\kappa\le\frac14\Bigl(1-\frac1\kappa\Bigr)^2\iff \gamma\le \frac14\Bigl(1+\frac1\kappa\Bigr)^2.
\]
This shows the case $\kappa<0$ in Theorem~\ref{theorem:lambert-tsallis},
and hence we have completed the proof of Theorem \ref{theorem:lambert-tsallis}.

\newpage

\section{Appendix}

\subsection{Tables}

This subsection contains increasing/decreasing tables of $H_\alpha$ and $b'(\theta)$.

\begin{Lemma}
\label{lemma:table_H}
One has the following increasing/decreasing table of $H_\alpha$.
\[
\begin{array}{|l||l|l|}
\hline
\ds {\rm(A)}\ 0<\alpha\le\frac12
&
\begin{array}{*{5}{c|}c}
x&0&\cdots&\frac{2\pi}{\alpha+1}&\cdots&2\pi\\ \hline
H'_\alpha&0&+&0&\multicolumn{2}{c}{+}\\ \hline
H_\alpha&0&\nearrow&H_\alpha(\frac{2\pi}{\alpha+1})&\searrow&H_\alpha(2\pi)
\end{array}
&
H_\alpha(2\pi)\ge 0\\ \hline\hline
\ds {\rm(B)}\ \frac12<\alpha<1
&
\begin{array}{*{5}{c|}c}
x&0&\cdots&\frac{2\pi}{\alpha+1}&\cdots&2\pi\\ \hline
H'_\alpha&0&+&0&\multicolumn{2}{c}{-}\\ \hline
H_\alpha&0&\nearrow&H_\alpha(\frac{2\pi}{\alpha+1})&\searrow&H_\alpha(2\pi)
\end{array}
&
H_\alpha(2\pi)< 0\\ \hline \hline
\ds {\rm(C)}\ 1<\alpha<2
&
\begin{array}{*{5}{c|}c}
x&0&\cdots&\frac{2\pi}{\alpha+1}&\cdots&\frac{2\pi}{\alpha}\\ \hline 
H'_\alpha&0&-&0&\multicolumn{2}{c}{+}\\ \hline
H_\alpha&0&\searrow&H(\frac{2\pi}{\alpha+1})&\nearrow&H_\alpha(\frac{2\pi}{\alpha})
\end{array}
&
H_\alpha(\tfrac{2\pi}{\alpha})>0\\ \hline\hline
\ds {\rm(D)}\ \alpha\ge 2
&
\begin{array}{*{5}{c|}c}
x&0&\cdots&\frac{2\pi}{\alpha+1}&\cdots&\frac{2\pi}{\alpha}\\ \hline
H'_\alpha&0&-&0&+&\\ \hline
H_\alpha&0&\searrow&H_\alpha(\frac{2\pi}{\alpha+1})&\nearrow&H_\alpha(\frac{2\pi}{\alpha})
\end{array}
&
H_\alpha(\frac{2\pi}{\alpha})\le 0\\ \hline
\end{array}
\]
\end{Lemma}

\begin{Lemma}
\label{lemma:table of b prime}
The signature of $b'(\theta)$ on the interval $I$ is given as follows.
\[
\begin{array}{|l|c|l|l|}
\hline
\multicolumn{4}{|c|}{\ds 0<\kappa<\tfrac{1}{2}}\\ \hline 
(1)&
    \begin{array}{*{5}{c|}c}
    \theta&0&\cdots&\varphi_*&\cdots&\pi\\ \hline
    b'(\theta)&&+&0&-&
    \end{array}
&
    \begin{array}{rl}
    {\rm(i)}& \ell(\kappa,a)\ge 0,\\
    {\rm(ii)}& \ell(\kappa,a)<0,\ -J_\kappa(a)>\frac1\kappa
    \end{array}
&
a>\frac{2\kappa^2+3\kappa+1}{6\kappa}
\\ \hline
(2)&
    b'(\theta)<0\text{ for }\theta\in I
&
    \begin{array}{rl}
    {\rm(i)}&\ell(\kappa,a)<0,\ -J_\kappa(a)\le \frac{1}{\kappa}
    \end{array}
&a\le \frac{2\kappa^2+3\kappa+1}{6\kappa}
\\ \hline \hline
\multicolumn{4}{|c|}{\kappa=\frac12}\\ \hline
(1)&
    b'(\theta)=0\text{ for }\theta\in I
&
    \ell(\kappa,a)=0
&
    a=1
\\ \hline
(2)&
    b'(\theta)>0\text{ for }\theta\in I
&
    \ell(\kappa,a)>0
&
    a>1
\\ \hline
(3)&
    b'(\theta)<0\text{ for }\theta\in I
&
    \ell(\kappa,a)<0
&
    a<1
\\ \hline \hline
\multicolumn{4}{|c|}{\frac12<\kappa<1}\\ \hline
(1)&
    \begin{array}{*{5}{c|}c}
    \theta&0&\cdots&\varphi_*&\cdots&\pi\\ \hline
    b'(\theta)&&-&0&+&
    \end{array}
&
    \begin{array}{rl}
    {\rm(i)}&\ell(\kappa,a)<0,\\
    {\rm(ii)}&\ell(\kappa,a)=0,\\
    {\rm(iii)}&\ell(\kappa,a)>0,\ -J_\kappa(a)>\frac1\kappa
    \end{array}
&
    a<\frac{2\kappa^2+3\kappa+1}{6\kappa}
\\ \hline
(2)&
    b'(\theta)>0\text{ for }\theta\in I
&
    -J_\kappa(a)\le \frac1\kappa,\ \ell(\kappa,a)>0
&
    a\ge \frac{2\kappa^2+3\kappa+1}{6\kappa}
\\ \hline \hline
\multicolumn{4}{|c|}{\kappa>1}\\ \hline
(1)&
    \begin{array}{*{5}{c|}c}
    \theta&0&\cdots&\varphi_*&\cdots&\frac\pi\kappa\\ \hline
    b'(\theta)&&+&0&-&
    \end{array}
&
    \ell(\kappa,a)>0,\ -J_\kappa(a)<\frac1\kappa
&
    a>\frac{2\kappa^2+3\kappa+1}{6\kappa}
\\ \hline
(2)&
    b'(\theta)<0\text{ for }\theta\in I
&
    \begin{array}{rl}
    {\rm(i)}&\ell(\kappa,a)<0,\\
    {\rm(ii)}&\ell(\kappa,a)=0,\\
    {\rm(iii)}&\ell(\kappa,a)>0,\ -J_\kappa(a)\ge \frac1\kappa
    \end{array}
&
    a\le \frac{2\kappa^2+3\kappa+1}{6\kappa}
\\ \hline

\end{array}
\]
If $\kappa=1$, then one has
\[
b'(\theta)=2(a-1)\sin\theta.
\]
\end{Lemma}

\newpage

\subsection{Figures}

This subsection collects figures of graphs of $f_{\kappa,\gamma}(x)$ (for real $x$), $F_\kappa$,
and of shapes of $\Omega$ with some deformation.

\begin{figure}[ht]
    \centering
    \begin{tabular}{cc}
    \begin{minipage}{0.5\textwidth}
        \centering
        \includegraphics[scale=0.25]{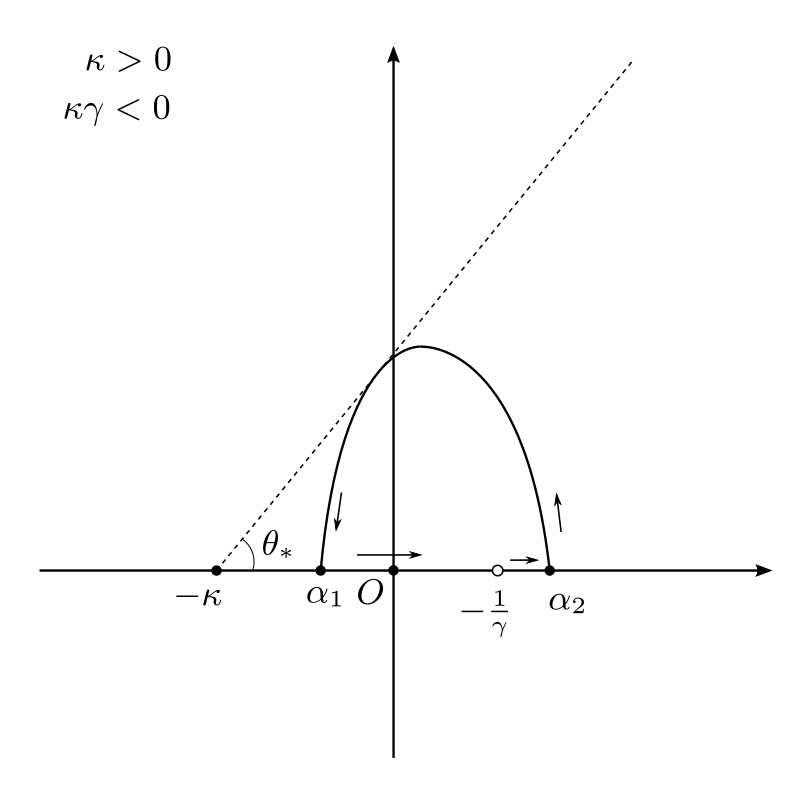}
        \caption{The case of (i)}
        \label{fig:path(i)}
    \end{minipage}
    &
    \begin{minipage}{0.5\textwidth}
        \centering
        \includegraphics[scale=0.25]{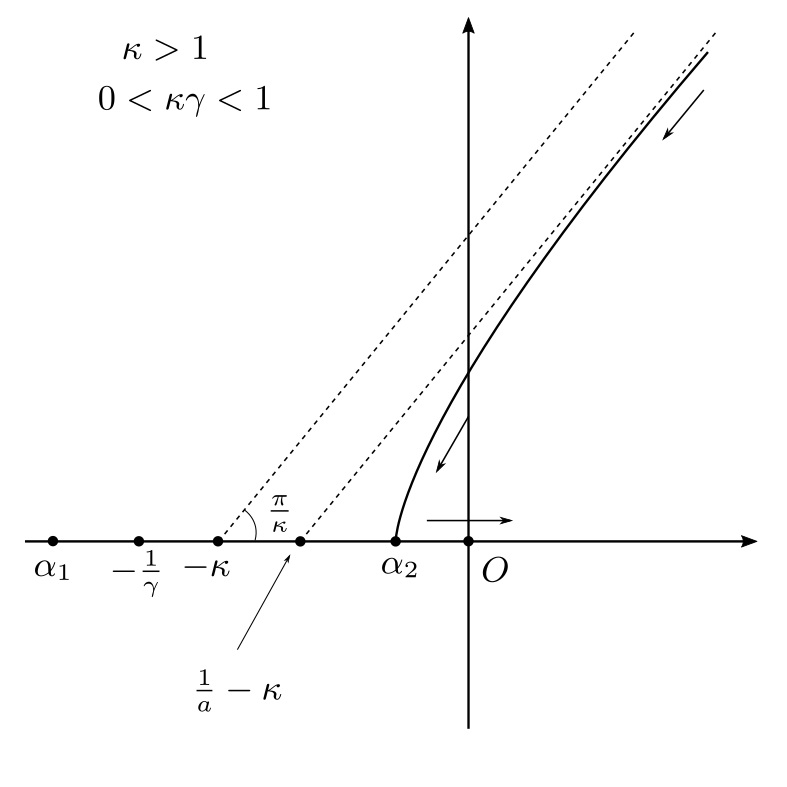}
        \caption{The case of (ii), {when $\kappa>2$}}
        \label{fig:path(ii)}
    \end{minipage}
    \\[10em]
    \begin{minipage}{0.5\textwidth}
        \centering
        \includegraphics[scale=0.25]{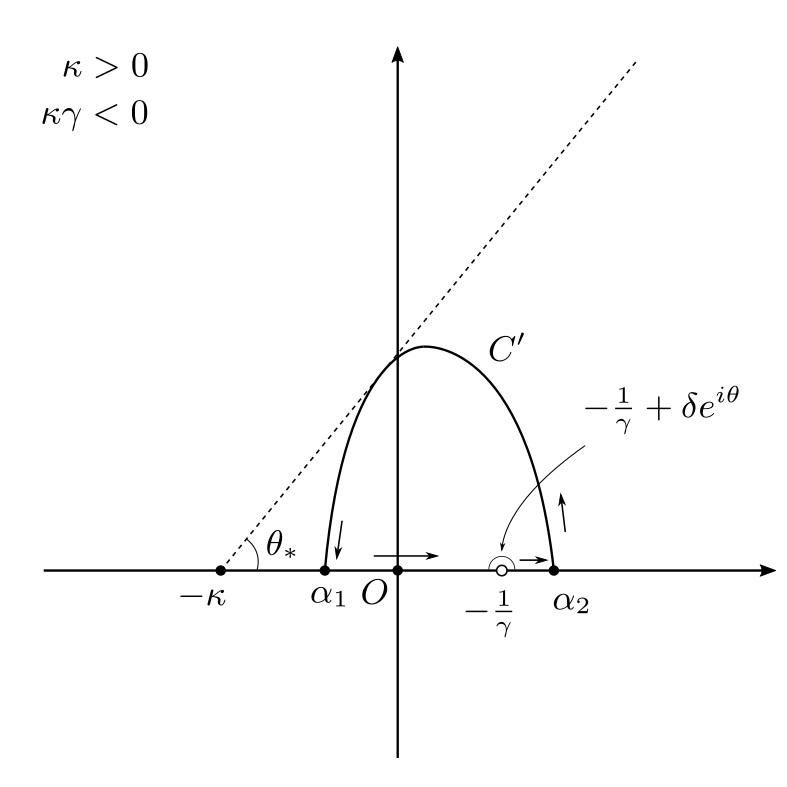}
        \caption{
        {Curve $C'$ in case (i)}}
        \label{fig:mpath(i)}
    \end{minipage}
    &
    \begin{minipage}{0.5\textwidth}
        \centering
        \includegraphics[scale=0.25]{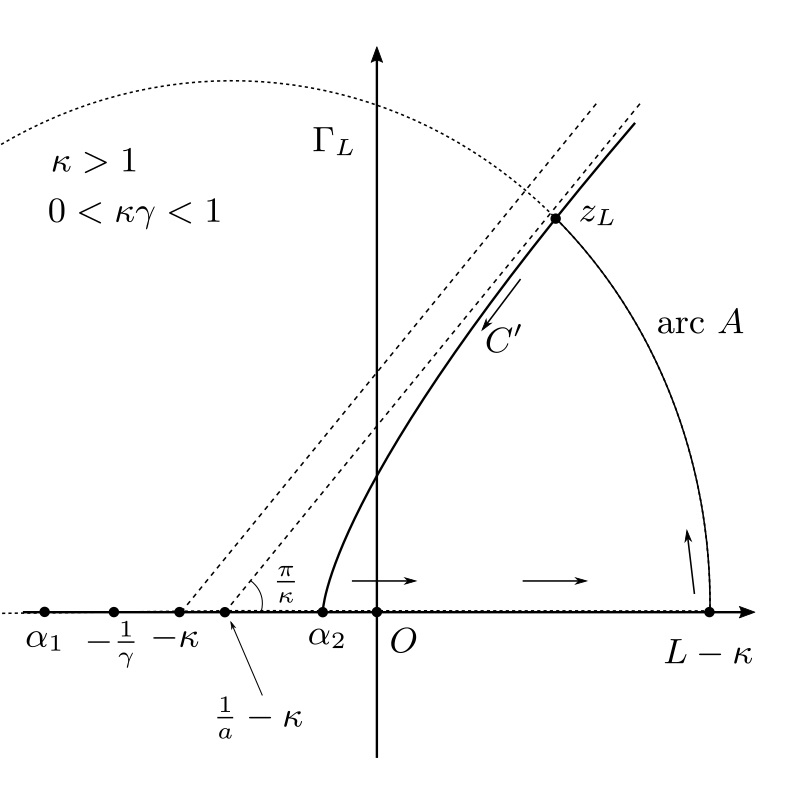}
        \caption{{Curve $C'$ in case (ii)}}
        \label{fig:mpath(ii)}
    \end{minipage}
    \\[10em]
    \begin{minipage}{0.5\textwidth}
        \centering
        \includegraphics[scale=0.25]{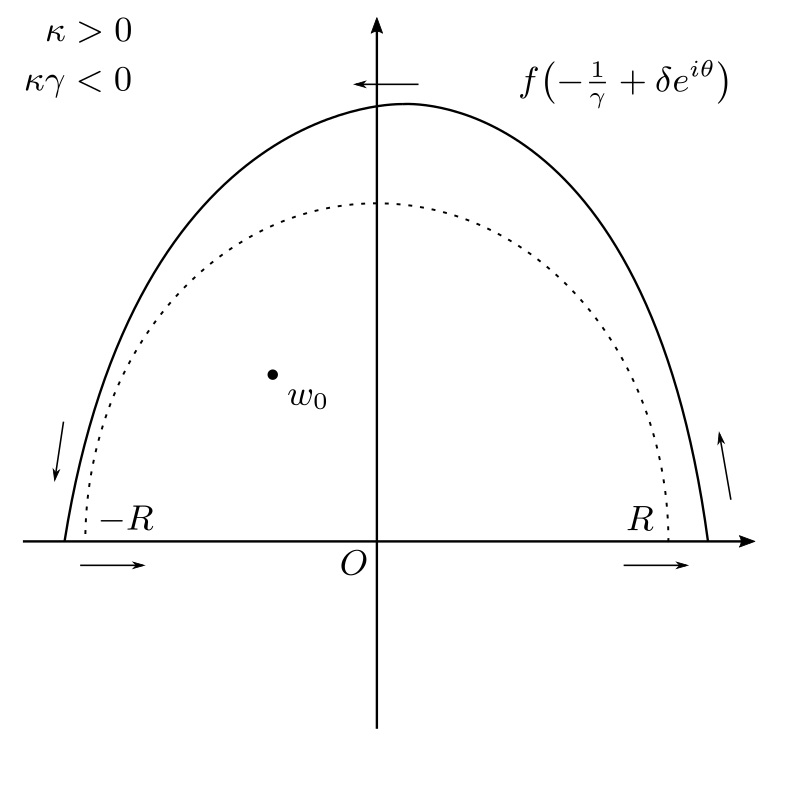}
        \caption{{Curve $f(C')$ in case (i)}}
        \label{fig:fpath(i)}
    \end{minipage}
    &
    \begin{minipage}{0.5\textwidth}
        \centering
        \includegraphics[scale=0.25]{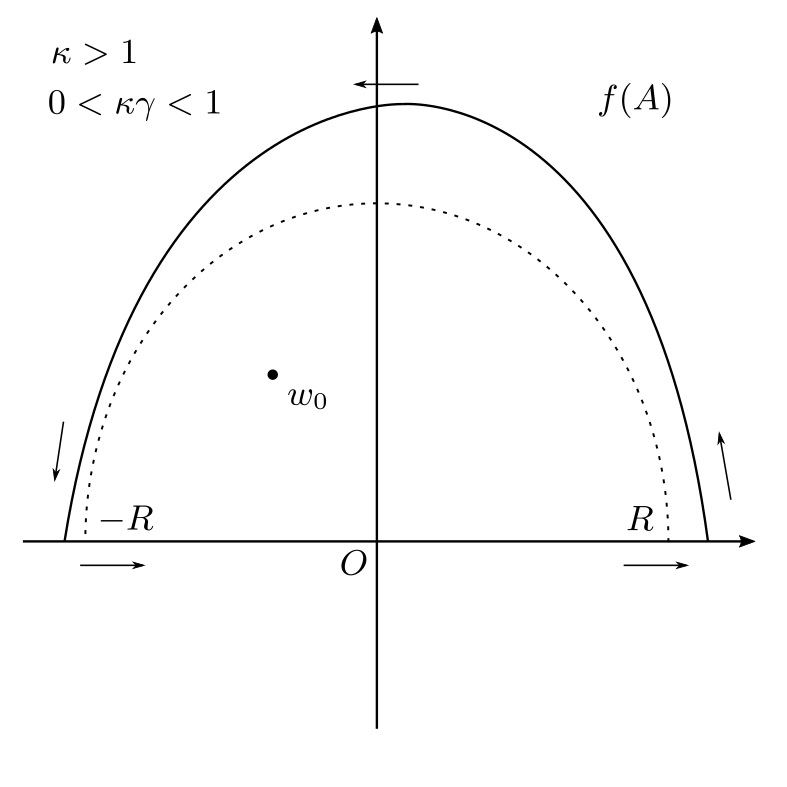}
       \caption{{Curve $f(C')$ in case (ii)}}
        \label{fig:fpath(ii)}
    \end{minipage}
    \end{tabular}
\end{figure}

\newpage

\begin{itemize}
\item Graphs of $f_{\kappa,\gamma}(x)$ for real $x$.
\begin{figure}[h]
    \centering
    \begin{tabular}{cc}
    \begin{minipage}{0.4\textwidth}
        \centering
        \includegraphics[scale=0.2]{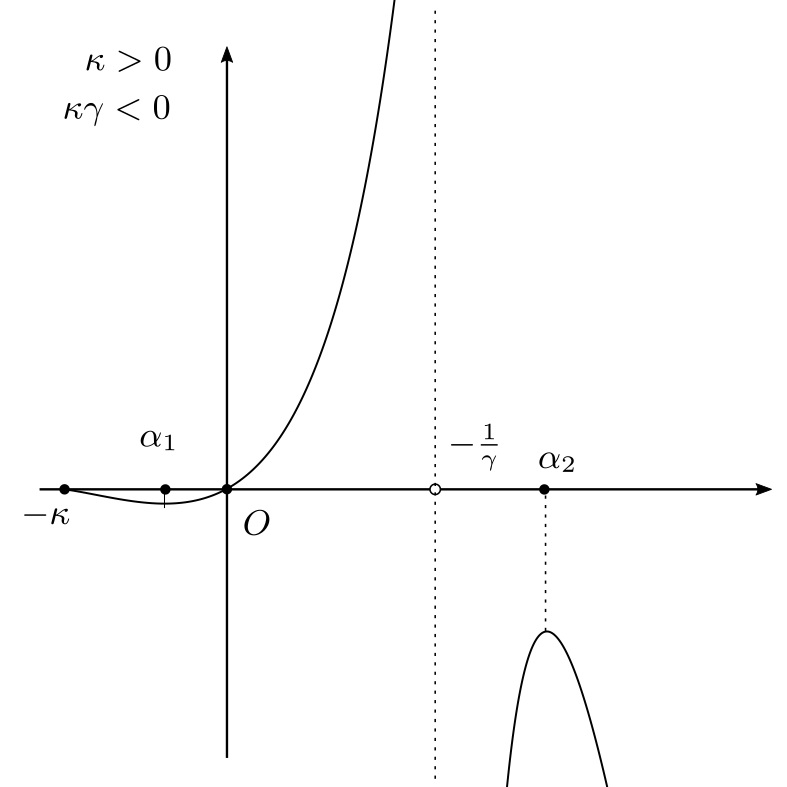}
        \caption{$f(x)$ for  $x\ge -\kappa$, case  (i)}
        \label{fig:graph(i)}
    \end{minipage}
    &
    \begin{minipage}{0.4\textwidth}
        \centering
       \includegraphics[scale=0.2]{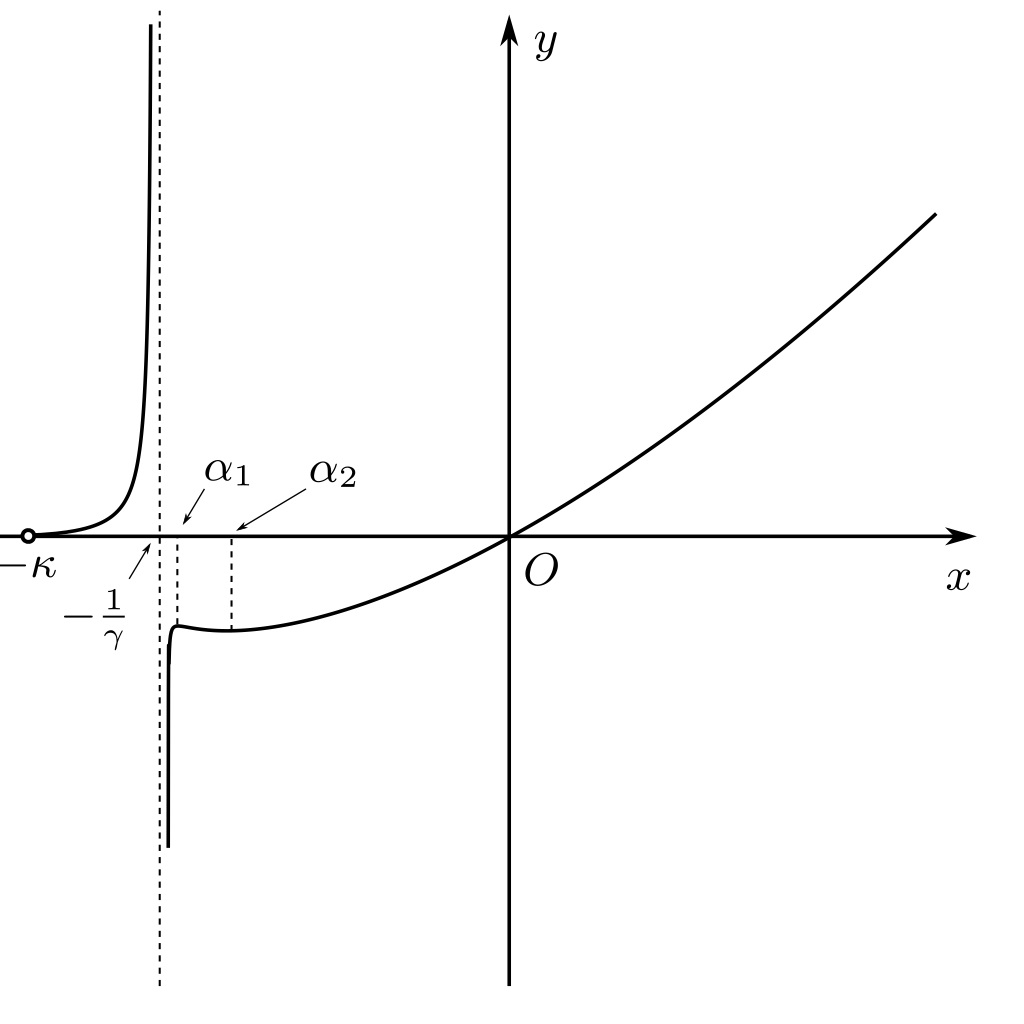}
    \end{minipage}
    \\
    \begin{minipage}{0.4\textwidth}
        \centering
        \includegraphics[scale=0.2]{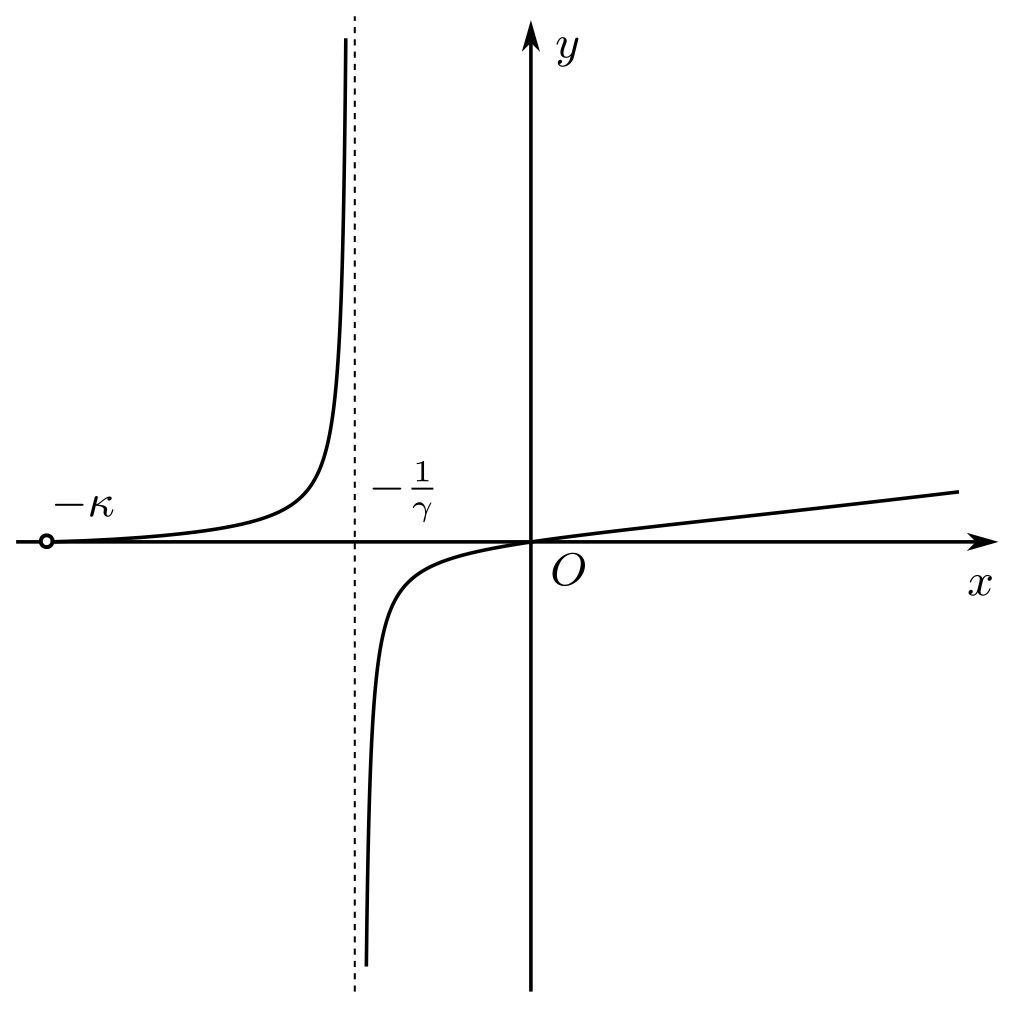}
    \end{minipage}
    &
    \begin{minipage}{0.4\textwidth}
        \centering
        \includegraphics[scale=0.2]{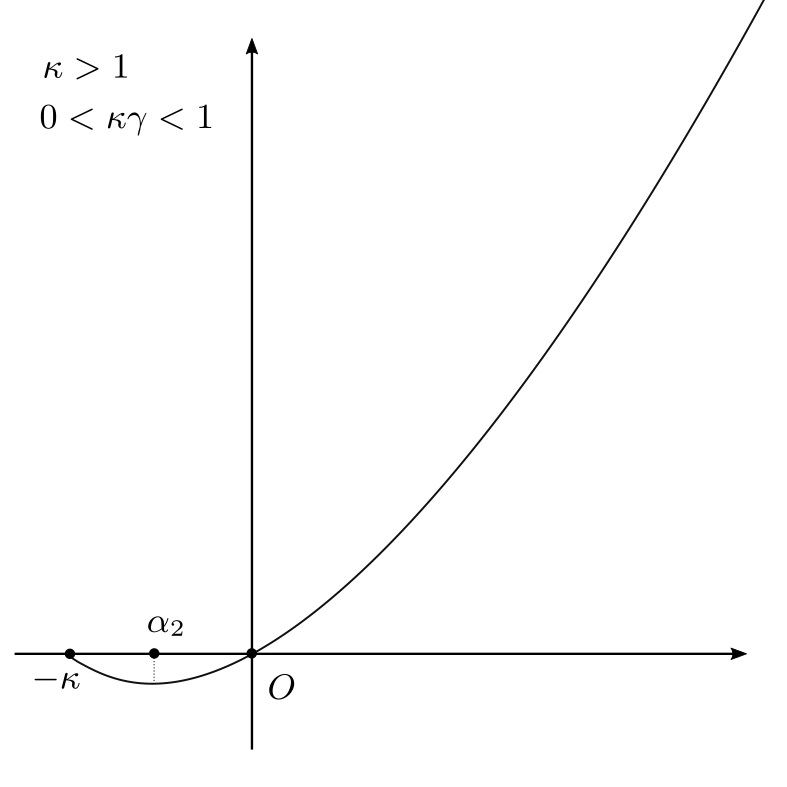}
        \caption{$f(x)$ for  $x\ge -\kappa$, case  (ii)}
        \label{fig:graph(ii)}
    \end{minipage}
    \end{tabular}
\end{figure}


\item Graphs of $F_\kappa$

\begin{center}
    \begin{tabular}{cc}
        $0<\kappa\le \frac12$&
        $\frac12<\kappa<1$\\
        \includegraphics[scale=0.35]{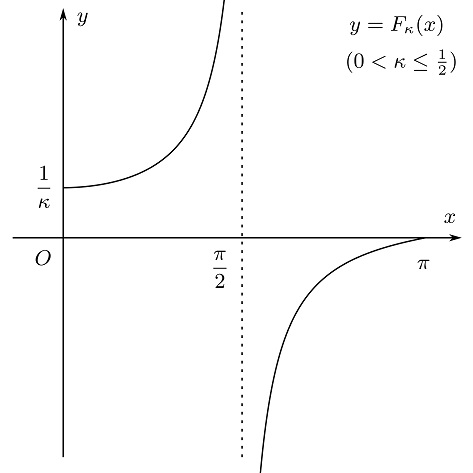}&
        \includegraphics[scale=0.35]{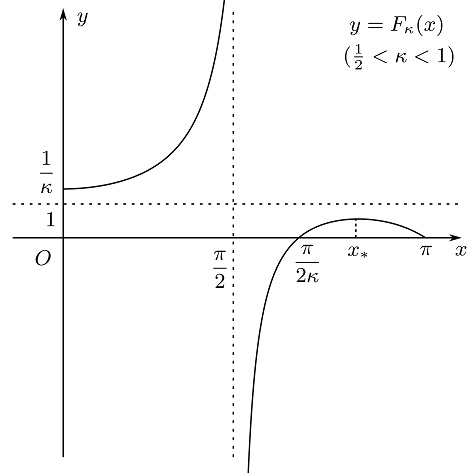}\\[1ex]
        $1<\kappa<2$&
        $\kappa\ge 2$\\
        \includegraphics[scale=0.35]{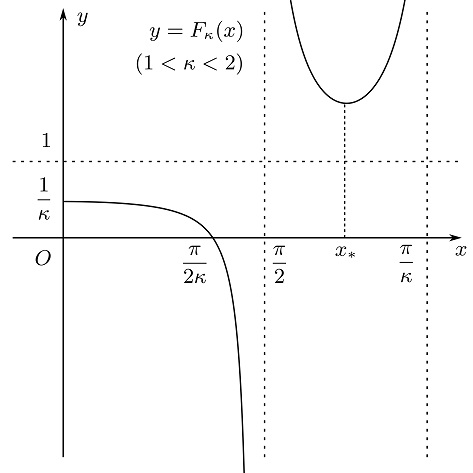}&
        \includegraphics[scale=0.35]{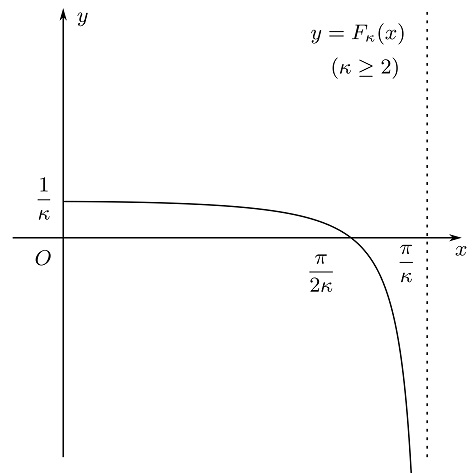}
    \end{tabular}
\end{center}


\newpage

\item Shapes of $\Omega$

\begin{center}
    \begin{tabular}{cc}
         The case of $\gamma<0$&The case of $\kappa>1$ and $0<\gamma<\frac1\kappa$\\
         \includegraphics[scale=0.4]{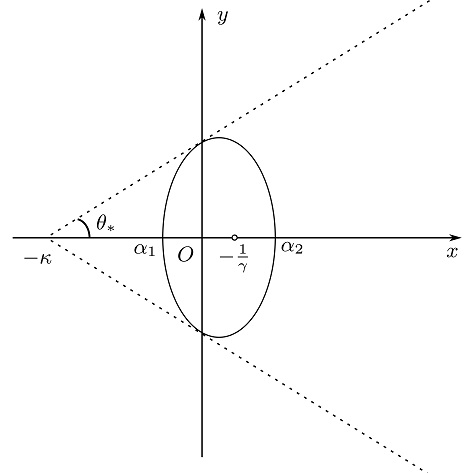}
         &
         \includegraphics[scale=0.4]{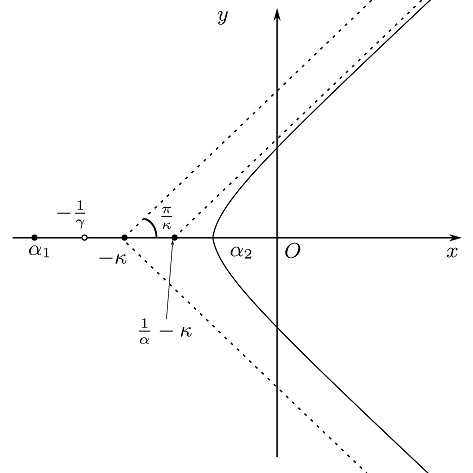}\\
         The case of $\kappa>1$, $a>1$ and $D(0)\ge 0$& The case of $\kappa>1$ and $D(0)<0$ \\
         \includegraphics[scale=0.4]{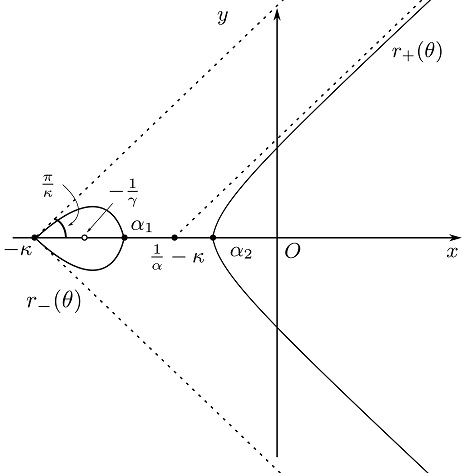}
         &
         \includegraphics[scale=0.4]{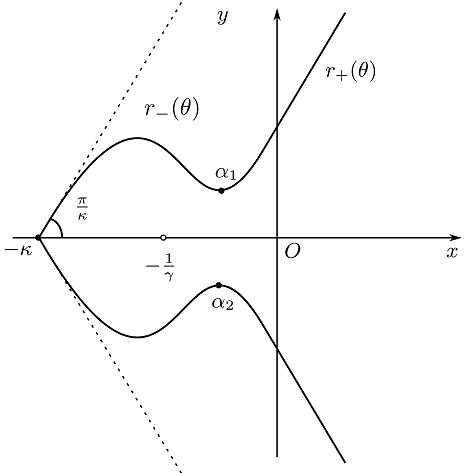}\\
         The case of $\kappa=1$ and $0<\gamma<\frac14$& 
         The case of $\kappa=\infty$ and $\gamma>\frac14$\\
         \includegraphics[scale=0.4]{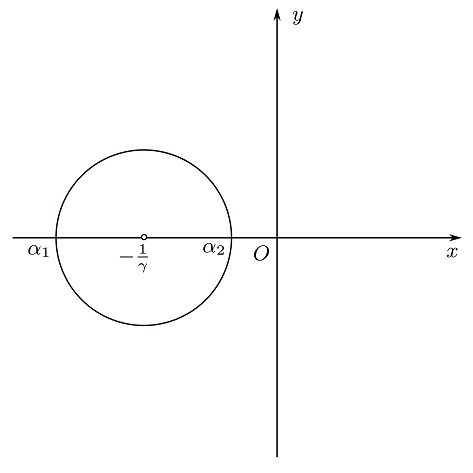}
         &
         \includegraphics[scale=0.4]{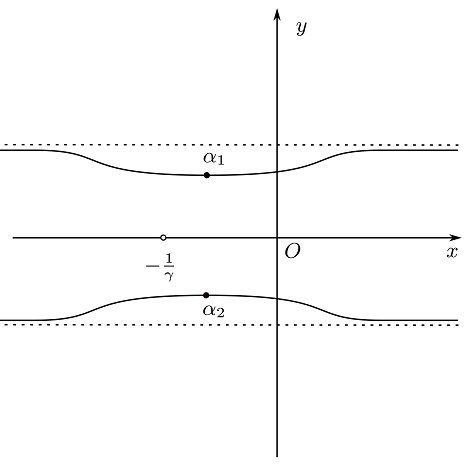}
    \end{tabular}
\end{center}
\end{itemize}


\begin{thebibliography}{99}
\bibitem{Ahlfors}
Ahlfors, L. (1979),
Complex Analysis, An introduction to the theory of analytic functions of one complex variable. Third edition. International Series in Pure and Applied Mathematics. McGraw-Hill Book Co., New York, 1978.
\bibitem{AmariOhara2011}
Amari, S., Ohara, A.\ (2011),
Geometry of $q$-exponential family of probability distributions,
Entropy \textbf{13}, no.\ 6, 1170--1185.
\bibitem{Borodin}
Borodin, A.\ (1999),
Biorthogonal ensembles, 
Nuclear Phys.\ B\textbf{536}, no.\ 3, 704–732.
\bibitem{Cheliotis}
Cheliotis, D.\ (2018),
Triangular random matrices and biorthogonal ensembles,
Statist.\ Probab.\ Letter \textbf{134}, 36--44.
\bibitem{CR}
Claeys, T., Romano, S.\ (2014),
Biorthogonal ensembles with two-particle interactions,
Nonlinearity \textbf{27}, no.\ 10, 2419--2443.
\bibitem{Corless}
Corless, R.\ M., Gonnet, G.\ H., Hare, D.\ E.\ G., 
Jeffrey, D.\ J., Knuth, D.\ E.\ (1996), 
On the Lambert $W$ function,
Adv.\ Comput.\ Math.\ 5, no.\ 4, 329–-359.
\bibitem{DykemaHaagerup}
Dykema, K., Haagerup, U.\ (2004),
DT-operator  and decomposability of Voiculescu's circular operator,
Amer.\ J.\ Math.\ \textbf{126}, 121--189.
\bibitem{NG2020_1}
Nakashima, H., Graczyk, P.,
Wigner and Wishart ensembles for graphical models, submitted, 30 pages.  arXiv:2008.10446
\bibitem{ZNS2018}
Zhang, F.\ D., Ng, H.\ K.\ T., Shi, Y.\ M. (2018),
Information geometry on the curved q-exponential family with application to survival data analysis,
Phys.\ A \textbf{512}, 788--802.
\end{thebibliography}
\end{document}